\documentclass[12pt,amstex]{amsart}

\usepackage{epsfig}
\usepackage{amsmath}
\usepackage{amssymb}
\usepackage{amscd}
\usepackage{graphicx}

\usepackage{pstricks}

\usepackage{tocvsec2}

\input xy
\xyoption{all}

\newtheorem{thm}{Theorem}[section]
\newtheorem{lem}[thm]{Lemma}

\newtheorem{prop}[thm]{Proposition}
\newtheorem{ques}[thm]{Question}
\newtheorem{cor}[thm]{Corollary}
\newtheorem{conj}[thm]{Conjecture}
\theoremstyle{definition}
\newtheorem{de}[thm]{Definition}

\theoremstyle{remark}
\newtheorem{rem}[thm]{Remark}

\numberwithin{equation}{section}

\def \N {\mathbb N}
\def \C {\mathbb C}
\def \Z {\mathbb Z}

\def \T {\mathbb{T}}

\def \F {\mathcal{F}}
\def \G {\mathcal{G}}
\def \I {\mathcal{I}}

\def \U {\mathcal{U}}

\def \X {\mathcal{X}}

\def \O {\mathcal{O}}

\def \Q {{\bf Q}}
\def \RP {{\bf RP}}

\def \id {{\rm id}}

\def \a {\alpha }

\def \ep {\epsilon}
\def \d {\delta}
\def \D {\Delta}
\def \ll {\lambda}

\def \lra{\longrightarrow}

\begin{document}

\title{Higher order almost automorphy, recurrence sets and the regionally proximal relation}

\author{Wen Huang}
\author{Song Shao}
\author{Xiangdong Ye}

\address{Department of Mathematics, University of Science and Technology of China,
Hefei, Anhui, 230026, P.R. China.}

\email{wenh@mail.ustc.edu.cn}\email{songshao@ustc.edu.cn}
\email{yexd@ustc.edu.cn}


\thanks{Huang is supported by  NNSF of China (10911120388), Fok Ying Tung Education
Foundation and the Fundamental Research Funds for the Central
Universities, Shao is supported by NNSF of China (10871186) and
Program for New Century Excellent Talents in University, and
Huang+Ye are supported by NNSF of China (11071231).}

\date{Oct. 12, 2011}
\begin{abstract}
In this paper, $d$-step almost automorphic systems are studied for
$d\in\N$, which are the generalization of the classical almost
automorphic ones.

For a minimal topological dynamical system $(X,T)$ it is shown that
the condition $x\in X$ is $d$-step almost automorphic can be
characterized via various subsets of $\Z$ including the dual sets of
$d$-step Poincar\'e and Birkhoff recurrence sets, and Nil$_d$
Bohr$_0$-sets by considering $N(x,V)=\{n\in\Z: T^nx\in V\}$, where
$V$ is an arbitrary neighborhood of $x$. Moreover, it turns out that
the condition $(x,y)\in X\times X$ is regionally proximal of order
$d$ can also be characterized via various subsets of $\Z$ including
$d$-step Poincar\'e and Birkhoff recurrence sets, $SG_d$ sets, the
dual sets of Nil$_d$ Bohr$_0$-sets, and others by considering
$N(x,U)=\{n\in\Z: T^nx\in U\}$, where $U$ is an arbitrary
neighborhood of $y$.

\end{abstract}

\maketitle

\markboth{Higher order almost automorphy etc.}{W. Huang, S. Shao and
X.D. Ye}


\section{Introduction}

In the past few years, it has become apparent both in ergodic theory
and additive combinatorics that nilpotent groups and a higher order
Fourier analysis play an important role. In this paper we will apply
results obtained by the same authors in \cite{HSY} to study higher
order automorphic systems, namely $d$-step almost automorphic
systems which by the definition are the almost one-to-one extensions
of their maximal $d$-step nilfactors. Since for a minimal system the
maximal $d$-step nilfactor is induced by the regionally proximal
relation of order $d$ (which is a closed invariant equivalence
relation \cite{HKM,SY}), the natural way we study $d$-step almost
automorphic systems is that we first get some characterizations of
regionally proximal relation of order $d$, and then obtain results
for $d$-step almost automorphic systems. In the process doing above
many interesting subsets of $\Z$ including higher order Poincar\'e
and Birkhoff recurrence sets (usual and cubic versions), higher
order Bohr sets, $SG_d$ sets (introduced in \cite{HK10}) and others
are involved. In this section we introduce the background and state
the main results of the paper.

\subsection{Background} First we give some background.

\subsubsection{Almost periodicity and almost automorphy}

The study of (uniformly) almost periodic functions was initiated by
Bohr in a series of three papers 1924-26 in \cite{Bohr}. The
literature on almost periodic functions is enormous, and the notion
has been generalized in several directions. Nowadays the theory of
almost periodic functions may be recognized as the representation
theory of compact Hausdorff groups: every topological group $G$ has
a group compactification $\a_G: G\rightarrow b G$ such that the
space of almost periodic functions on $G$ is just the set of all
functions $f\circ \a_G$ with $f\in C(b G)$. The compactification
$(\a_G, b G)$ of $G$ is called the {\em Bohr compactification} of
$G$.

Related to the almost periodic functions are the {\em almost
automorphic functions}: these functions turn out to be the ones of
the form $h\circ \a_G$ with $h$ a bounded continuous function on
$\a_G(G)$ ( if $h$ is uniformly continuous and bounded on $\a_G(G)$,
then it extends to an $f\in C(bG)$, so $h\circ \a_G=f\circ \a_G$ is
almost periodic on $G$).

\medskip

The notion of almost automorphy was first introduced by Bochner in
1955 in a work of differential geometry \cite{Bochner55, Bochner62}.
Taking $G$ for the present to be the group of integers $\Z$ and an
almost automorphic function $f$ has the property that from any
sequence $\{n_i'\}\subseteq \Z$ one may extract a subsequence
$\{n_i\}$ such that both
\begin{equation*}
    \lim_{i\to \infty} f(t+n_i)=g(t)\quad \text{and }\quad \lim_{i\to \infty}
    g(t-n_i)=f(t)
\end{equation*}
hold for each $t\in \Z$ and some function $g$, not necessarily
uniformly. Bochner  \cite{Bochner62} has observed that almost
periodic functions are almost automorphic, but the converse is not
true. Veech \cite{V65} showed that the almost automorphic functions
can be characterized in terms of the almost periodic ones, and vice
versa. In the same paper, Veech considered the system associated
with an almost automorphic function, and introduced the notion of
{\em almost automorphic point} ({\em AA point}, for short) in
topological dynamical systems (t.d.s. for short).
For a t.d.s. $(X, T)$, a point $x\in X$ is said to be {\em almost
automorphic} if from any sequence $\{n_i'\}\subseteq \Z$ one may
extract a subsequence $\{n_i\}$ such that
$$\lim_{j\to \infty}\lim_{i\to \infty} T^{n_i-n_j}x=x.$$
Moreover, Veech  \cite{V65, V68} gave the structure theorem for
minimal systems with an almost automorphic point: each minimal
almost automorphic system is an almost one-to-one extension of its
maximal equicontinuous factor.

Note that in \cite{V65} all works were done for general groups. The
notion of almost automorphy is very useful in the study of
differential equations, and  see \cite{ShenYi} and references there
for more information on this topic.

\subsubsection{The equicontinuous structure relation $S_{eq}$, almost
automorphy and Bohr$_0$ sets}

For a t.d.s. $(X,T)$, it was proved in \cite{EG} that there exists
on $X$ a closed $T$-invariant equivalence relation, $S_{eq}$, such
that $(X/S_{eq}, T)$ is an equicontinuous system. $S_{eq}$ is called
the {\em equicontinuous structure relation}. It was also showed in
\cite{EG} that $S_{eq}$ is the smallest closed $T$-invariant
equivalence relation containing the regionally proximal relation
$\RP=\RP(X)$ (recall that $(x,y)\in \RP$ if there are sequences
$x_i,y_i\in X, n_i\in \Z$ such that $x_i\to x, y_i\to y$ and
$(T\times T)^{n_i}(x_i,y_i)\to (z,z)$, $i\to \infty$, for some $z\in
X$). A natural question was whether $S_{eq}=\RP(X)$ for all minimal
t.d.s.? Veech \cite{V68} gave the first positive answer to this
question, i.e. he proved that $S_{eq}=\RP(X)$ for all minimal t.d.s.
under abelian group actions. As a matter of fact, Veech proved that
for a minimal t.d.s. $(x,y)\in S_{eq}$ if and only if there is a
sequence $\{n_i\}\subset \Z$ and $z\in X$ such that
\begin{equation*}
    T^{n_i}x\lra z \quad \text{and}\quad T^{-n_i}z\lra y, \ i\to \infty.
\end{equation*}
As a direct corollary, for a minimal t.d.s. $(X,T)$, {\em a point
$x\in X$ is almost automorphic if and only if }$$\RP[x]=\{y\in X:
(x,y)\in \RP\}=\{x\}.$$

Also from Veech's approach, it is easy to show that for a minimal
t.d.s. $(X,T)$, $(x,y)\in \RP$ if and only if for each neighborhood
$U$ of $y$, $N(x,U)=\{n\in\Z: T^nx\in U\}$ contains some
$\D$-set\footnote{A $\D$-set is a set of differences $A-A=\{a-b:
a,b\in A\}$ for some infinite subset $A\subset \Z$; and a $\D^*$-set
is a set that has nontrivial intersection with the set of $A-A$ for
any infinite set $A$. }. Hence it is not difficult to get another
equivalent condition for an almost automorphic point \cite{F}: {\em
a point $x\in X$ is almost automorphic if and only if it is
$\D^*$-recurrent.}\footnote{ Let $\F$ be a collection of subsets of
$\Z$ and let $(X,T)$ be a t.d.s.. A point $x$ of $X$ is called {\em
$\F$-recurrent} if $N(x,U)\in \F$ for every neighborhood $U$ of
$x$.}

\medskip

Recall a subset $A\subseteq \Z$ is a {\em Bohr$_0$ set} if there
exists an equicontinuous system $(X,T)$, a point $x_0\in X$ and its
open neighborhood $U$ such that $N(x_0,U)=\{n\in \Z: T^nx_0\in U\}$
is contained in $A$.\footnote{There are lots of equivalent
definitions for Bohr set. For example, one may define Bohr sets as
follows: A subset $A\subseteq \Z$ is a {\em Bohr set} if there exist
$m\in \N$, $\a\in \T^m$, and an open set $U\subseteq \T^m$ such that
$    \{n\in \Z: n\a \in U\}$ is contained in $A$; the set $A$ is a
{\em Bohr$_0$ set} if additionally $0\in U$. See \cite{BFW, Kaz} for
more details.} Since every point in an equicontinuous system is
almost automorphic, it follows that each Bohr$_0$ set is a
$\D^*$-set. The converse does not hold \cite{BFW}. But a $\D^*$-set
is not too far from being a Bohr$_0$-set. It is shown by Host and
Kra recently that each $\D^*$-set is a piecewise Bohr$_0$-set,
meaning that it agrees with a Bohr$_0$-set on a sequence of
intervals whose lengths tend to infinity \cite{HK10}.

\subsubsection{Poincar\'e recurrence sets and almost automorphy}

The Birkhorff recurrence theorem states that each t.d.s. has a
recurrent point which implies that whenever $(X, T)$ is a minimal
t.d.s. and $U\subseteq X$ a nonempty open set, then $N(U,U)\neq
\emptyset$. The measurable version of this phenomenon is the famous
Poincar\'e's Recurrence Theorem: Let $(X,\X,\mu,T)$ be a measure
preserving system and $A\in \X$ with $\mu(A)>0$, then
$N_\mu(A,A)=\{n\in \Z: \mu(A\cap T^{-n}A)>0\}$ is infinite.

\medskip
In \cite{F, F81} Furstenberg introduced the notion of Poincar\'e and
Birkhoff recurrence sets. A subset $P$ of $\Z$ is called a {\em
Poincar\'e recurrence set} if whenever $(X,\X,\mu,T)$ is a measure
preserving system and $A\in \X$ has positive measure, then $P\cap
N_\mu(A,A)\neq \emptyset $. Similarly, a subset $P\subset \Z$ is
called a {\em Birkhoff recurrence set} if whenever $(X, T)$ is a
minimal t.d.s. and $U\subseteq X$ a nonempty open set, then $P\cap
N(U,U)\neq \emptyset$. Let $\F_{Poi}$ and $\F_{Bir}$ denote the
collections of Poincar\'e  and Birkhoff recurrence sets of $\Z$
respectively.

\medskip

In \cite{HLY}, it was shown for a minimal t.d.s. $(x,y)\in \RP$ if
and only if for each neighborhood $U$ of $y$, $N(x,U)\in \F_{Poi}$.
We will show that one can use $\F_{Poi}$ to get another equivalent
condition for an almost automorphic point: {\em a point $x\in X$ is
almost automorphic if and only if it is $\F_{Poi}^*$-recurrent,}
where $\F_{Poi}^*$ is the collection of subsets of $\Z$ intersecting
all sets from $\F_{Poi}$. One has similar results for Birkhoff
recurrence sets.



\subsubsection{Multiple ergodic averages and factors}

It is stated by Von Neumann and Birkhoff ergodic theorems that
ergodic average $ \frac 1N \sum_{n=0}^{N-1} f(T^n x)$ converges in
$L^2$ and pointwisely respectively. The study of the multiple
ergodic averages
\begin{equation*}
    \frac 1 N\sum_{n=0}^{N-1}f_1(T^nx)\ldots f_d(T^{dn}x)
\end{equation*}
begins from the Furstenberg's beautiful proof of Szemer\'edi's
theorem via ergodic theory \cite{F77} in the 1970's. After nearly 30
years' efforts of many researchers, this problem of $L^2$ case was
finally solved by Host and Kra in \cite{HK05} (see also Ziegler
\cite{Z}). In their proofs the theory of nilfactors plays a great
role. The structure theorem of \cite{HK05, Z} states that if one
wants to understand the multiple ergodic averages
$$ \frac 1 N\sum_{n=0}^{N-1}f_1(T^nx)\ldots f_d(T^{dn}x) ,$$
one can replace each function $f_i$ by its conditional expectation
on its $d-1$-step nilfactor (a $1$-step nilfactor is the Kroneker's
one). Thus one can reduce the problem to the study of the same
average in a nilsystem.

The study of the topological correspondence of the nilfactors has a
long history. It goes back to the study of the equicontinuous
structure relation $S_{eq}(X)$ of a t.d.s. $(X, T)$ in the 1960's,
and more recently Glasner's work \cite{G93, G94} etc.. It turns out
the notion of the regionally proximal relation of order $d$ defined
in \cite{HM,HKM} plays an important role.

\begin{de}
Let $(X, T)$ be a t.d.s.  and let $d\ge 1$ be an integer. A pair
$(x, y) \in X\times X$ is said to be {\em regionally proximal of
order $d$} if for any $\d  > 0$, there exist $x', y'\in X$ and a
vector ${\bf n} = (n_1,\ldots , n_d)\in\Z^d$ such that $\rho(x, x')
< \d, \rho(y, y') <\d$, and $$ \rho(T^{{\bf n}\cdot \ep}x', T^{{\bf
n}\cdot \ep}y') < \d\ \text{for any $\ep\in \{0,1\}^d$,
$\ep\not=(0,\ldots,0)$},
$$ where ${\bf n}\cdot \ep = \sum_{i=1}^d \ep_in_i$. The set of
regionally proximal pairs of order $d$ is denoted by $\RP^{[d]}(X)$,
which is called {\em the regionally proximal relation of order $d$}.
\end{de}
It is easy to see that $\RP^{[d]}(X)$ is a closed and invariant
relation for all $d\in \N$. When $d=1$, $\RP^{[d]}(X)$ is nothing
but the classical regionally proximal relation. In \cite{HKM}, for a
minimal distal t.d.s. the authors showed that $\RP^{[d]}(X)$ is a
closed invariant equivalence relation, and the quotient of $X$ under
this relation is its maximal $d$-step nilfactor. These results were
proved to be true for general minimal t.d.s. \cite{SY}.

\subsubsection{Nilsystems and nilsequences}

Furstenberg's proof of Szemer\'edi's
theorem via ergodic theory paved the way for new combinatorial
results via ergodic methods, as well as leading to numerous
developments within ergodic theory. More recently, the interaction
between the fields has taken a new dimension, with ergodic objects
being imported into the finite combinatorial setting. Some objects
at the center of this interchange are nilsequences and the
nilsystems on which they are defined (see, for example,
\cite{BHK05,GT08, GT, GT10, HK05, HK09, HK10, HKM}).

\medskip
Nilsequences are defined by evaluating a function along the orbit of
a point in the homogeneous space of a nilpotent Lie group. We recall
the definition of a nilsequence. A {\em basic $d$-step nilsequence}
is a sequence of the form $\{f (T^ nx): n \in  \Z\}$, where $d\in
\N$ and $(X, T )$ is a basic $d$-step nilsystem, $f :X\rightarrow
\C$ is a continuous function, and $x\in X$. A {\em $d$-step
nilsequence} is a uniform limit of basic $d$-step nilsequences.

\medskip

\medskip

One can define a generalization of a Bohr$_0$ set \cite{HK10}:

\begin{de}
A subset $A\subseteq \Z$ is a {\em Nil$_d$ Bohr$_0$-set} of there
exist a $d$-step nilsystem $(X,T)$, $x_0\in X$ and an open set
$U\subseteq X$ containing $x_0$ such that
\begin{equation*}
    \{n\in \Z: T^n x_0\in U\}
\end{equation*}
is contained in $A$.
\end{de}

Denote by $\F_{Bohr_0}$ and $\F_{d,0}$ the family generated by all
Bohr$_0$-sets and Nil$_d$ Bohr$_0$-sets respectively. Note that
$\F_{Bohr_0}=\F_{1,0}$.

\subsubsection{$d$-step almost automorphy}

Similar to the definition of almost automorphy, now we give the
definition of $d$-step almost automorphy for all $d\in \N$:

\begin{de}
Let $(X,T)$ be a minimal t.d.s. and $x\in X$, $d\in \N$. $x$ is
called {\em $d$-step almost automorphic} (or $d$-step AA for short)
if $\RP^{[d]}[x]=\{x\}$. A minimal t.d.s.  is called {\em $d$-step
almost automorphic} if it has a $d$-step almost automorphic point.
\end{de}

Since $\RP^{[d]}$ is an equivalence relation for minimal t.d.s.
\cite{SY}, by definition it follows that

\begin{prop}
Let $(X,T)$ be a minimal t.d.s.. Then $(X,T)$ is a $d$-step almost
automorphic system for some $d\in \N$ if and only if it is an almost
one-to-one extension of its maximal $d$-step nilfactor.
\end{prop}

\subsubsection{Higher order recurrence sets}

In this paper, we will use recurrence sets to characterize $d$-step
almost automorphy. First we need to generalize the recurrence sets
to a higher order version.

\medskip

Before doing this we state the multiple Poincar\'e and Birkhoof
recurrence theorems, see \cite{F}

\noindent $\bullet$  Let $(X,\X,\mu,T)$ be a measure preserving
system and $d\in\N$. Then for any $A\in \X$ with $\mu(A)>0$ there is
$n\in\Z\setminus \{0\}$ such that $\mu(A\cap T^{-n}A\cap \ldots\cap
T^{-dn}A)>0$.

\noindent $\bullet$ Let $(X,T)$ be a t.d.s. and $d>0$. Then there
are $x\in X$ and a subsequence $\{n_i\}$ with $n_i\lra +\infty$ such
that $\lim_{i\lra+\infty}T^{jn_i}x=x$ for each $1\le j\le d$.

The facts enable us to get generalizations of Poincar\'e and
Birkhoff recurrence subsets (see \cite{FLW}).

\begin{de} Let $d\in \N$.
\begin{enumerate}
\item We say that $S \subset \Z$ is a set of {\em $d$-recurrence } if
for every measure preserving system $(X,\X,\mu,T)$ and for every
$A\in \X$ with $\mu (A)
> 0$, there exists $n \in S$  such that
$$\mu(A\cap T^{-n}A\cap \ldots \cap T^{-dn}A)>0.$$


\item We say that $S\subset \Z$ is a set of {\em $d$-topological
recurrence} if for every minimal t.d.s. $(X, T)$ and for every
nonempty open subset $U$ of $X$, there exists $n\in S$ such that
$$U\cap T^{-n}U\cap \ldots \cap T^{-dn}U\neq \emptyset.$$
\end{enumerate}
\end{de}

Let $\F_{Poi_d}$ (resp. $\F_{Bir_d}$) be the family generated by the
collection of all sets of $d$-recurrence (resp. sets of
$d$-topological recurrence). It is obvious by definitions that
$\F_{Poi_d}\subset \F_{Bir_d}$. It is showed in \cite{HSY} that
these sets are contained in the dual family of Nil$_d$-Bohr$_0$
sets.

\begin{prop}\cite{HSY}
Let $d\in\N$. Then
$$\F_{Poi_d}\subset\F_{Bir_d}\subset \F^*_{d,0},$$ where
$\F_{d,0}^*$ is the collection of subsets of $\Z$ intersecting all
Nil$_d$ Bohr$_0$ sets.
\end{prop}

Note that $\F_{Poi}\neq \F_{Bir}$ \cite{K}. Though we can not prove
if $\F_{Bir_d} = \F^*_{d,0}$, we will show that we can not
distinguish them in the dynamical sense (Theorem \ref{intro-10}).

\begin{rem} The above definitions are slightly different from the
ones introduced in \cite{FLW}, namely we do not require $n\not=0$.
The main reason we define in this way is that for each $A\in
\F_{d,0}$, $0\in A$. Thus $\{0\}\cup C\in \F^*_{d,0}$ for each
$C\subset \Z$.
\end{rem}

\subsection{Main results} Now we are ready to state the main
results.

\subsubsection{Regionally proximal relation of order $d$ and $d$-step
almost automorpy}

The following theorem shows that we can use $\F_{Poi_d}$,
$\F_{Bir_d}$ and $\F_{d,0}^*$ to characterize regionally proximal
pairs of order $d$.

\begin{thm}\label{intro-10}
Let $(X,T)$ be a minimal t.d.s.. Then the following statements are
equivalent:

\begin{enumerate}

\item $(x,y)\in \RP^{[d]}$.

\item
$N(x,U)\in \F_{Poi_d}$ for each neighborhood $U$ of $y$.

\item $N(x,U)\in \F_{Bir_d}$ for each
neighborhood $U$ of $y$.

\item $N(x,U)\in \F_{d,0}^*$ for each
neighborhood $U$ of $y$.
\end{enumerate}
\end{thm}

Using the Ramsey property of the families, we can show that one can
use $\F_{Poi_d}^*$, $\F_{Bir_d}^*$ and $\F_{d,0}$ to characterize
$d$-step almost automorphy.

\begin{thm}\label{intro-20}
Let $(X,T)$ be a minimal t.d.s. and $d\in\N$. Then the following
statements are equivalent
\begin{enumerate}
\item $(X,T)$ is d-step almost automorphic.

\item There is $x\in X$ such that $N(x,V)\in \F_{Poi_d}^*$ for each
neighborhood $V$ of $x$.

\item There is $x\in X$ such that $N(x,V)\in \F_{Boi_d}^*$ for each
neighborhood $V$ of $x$.

\item There is $x\in X$ such that $N(x,V)\in \F_{d,0}$ for each
neighborhood $V$ of $x$.

\end{enumerate}
\end{thm}

\subsubsection{$d$-step almost automorphy and $SG_d$-sets}

In this paper, we also discuss $SG_d$-sets introduced by Host and
Kra recently \cite{HK10} and show that one may use it to
characterize regionally proximal pairs of order $d$.

\medskip

Let $d\ge 1$ be an integer and let $P=\{p_i\}_i$ be a (finite or
infinite) sequence in $\Z$. The {\em set of sums with gaps of length
less than $d$} of $P$ is the set $SG_d(P)$ of all integers of the
form
$$\ep_1p_1+\ep_2p_2+\ldots +\ep_np_n$$ where $n\ge 1$ is an integer,
$\ep_i\in \{0,1\}$ for $1\le i\le n$, the $\ep_i$ are not all equal
to $0$, and the blocks of consecutive $0$'s between two $1$ have
length less than $d$. A subset $A \subseteq \Z$ is an $SG_d$-set if
$A=SG_d(P)$ for some infinite sequence of $\Z$; and it is an
$SG^*_d$-set if $A \cap SG_d(P)\neq \emptyset$ for every infinite
sequence $P$ in $\Z$. Let $\F_{SG_d}$ be the family generated by all
$SG_d$-sets.  Note that each $SG_1$-set is a $\D$-set, and each
$SG_1^*$-set is a $\D^*$-set.

\medskip

The following is the main result of \cite{HK10}

\begin{prop}[Host-Kra]
Every $SG_d^*$-set is a PW-Nil$_d$ Bohr$_0$-set.
\end{prop}

Host and Kra \cite{HK10} asked the following

\begin{ques}\label{Ques-HK}
Is every Nil$_d$ Bohr$_0$-set an $SG_d^*$-set?
\end{ques}

We have

\begin{thm}\label{intro-11}
Let $(X,T)$ be a minimal t.d.s., $x,y\in X$, and $d\in \N$. Then
$(x,y)\in \RP^{[d]}$  if and only if $N(x,U)\in \F_{SG_d}$ for each
neighborhood $U$ of $y$.
\end{thm}

Combining Theorems \ref{intro-10} and \ref{intro-11} we see that
Nil$_d$ Bohr$_0$-sets and $SG_d^*$-sets are closely related. A
direct corollary of Theorem \ref{intro-11} is: let $(X,T)$ be a
minimal t.d.s., $x\in X$, and $d\in \N$. If $x$ is
$\F^*_{SG_d}$-recurrent, then it is $d$-step almost automorphic. We
have the following conjecture.
\begin{conj}
Let $(X,T)$ be a minimal t.d.s., $x\in X$, and $d\in \N$. Then $x$
is $d$-step almost automorphic if and only if it is
$SG^*_d$-recurrent.
\end{conj}

Since $SG_d$-sets do not have the Ramsey property (Appendix
\ref{appendix:Ramsey}), we can not apply the methods in the proof of
Theorem \ref{intro-20} to show the above conjecture. Note that if
Question \ref{Ques-HK} has a positive answer, then by using Theorem
\ref{intro-20} the above conjecture holds.

\subsubsection{Cubic version of multiple Poincar\'e recurrence sets
}

One can also characterize the higher order regionally proximal
relation via cubic version of multiple Poincar\'e and Birkhoff
recurrence sets. For $d\in \N$, a subset $F$ of $\Z$ is a {\em
Poincar\'e recurrence set of order $d$} if for each measure
preserving system $(X,\mathcal{B},\mu,T)$ and $A\in \mathcal{B}$
with positive measure there are $n_1,\ldots,n_d\in\Z$ such that
$FS(\{n_i\}_{i=1}^d)=\{n_{i_1}+\cdots+n_{i_k}:1\leq
i_1<\cdots<i_k\leq d\}\subset F$ and
$$\mu\Big(A\cap\big (\bigcap _{n\in
FS(\{n_i\}_{i=1}^d)} T^{-n}A\big )\Big)>0.$$ Similarly, we define
Birkhoff recurrence sets of order $d$. Let for $d\in \N$, $\F_{P_d}$
(resp. $\F_{B_d}$) be the family of all Poincar\'e recurrence sets
of order $d$ (resp. the family of all Birkhoff recurrence sets of
order $d$).

Via recurrence sets of order $d$, we have the following result:

\begin{thm}\label{intro-12}
Let $(X,T)$ be a minimal t.d.s. and $x,y\in X$, $d\in \N$. Then the
following statements are equivalent:
\begin{enumerate}
\item
 $(x,y)\in \RP^{[d]}$.

\item  $N(x,U)\in \F_{P_d}$ for each
neighborhood $U$ of $y$.

\item  $N(x,U)\in \F_{B_d}$ for each neighborhood $U$ of $y$.
\end{enumerate}
\end{thm}

A direct corollary of Theorem \ref{intro-12} is: let $(X,T)$ be a
minimal t.d.s., $x\in X$, and $d\in \N$. If $x$ is
$\F^*_{P_d}$-recurrent, or $\F^*_{B_d}$-recurrent then it is
$d$-step almost automorphic. We have the following conjecture.

\begin{conj}
Let $(X,T)$ be a minimal t.d.s., $x\in X$, and $d\in \N$. Then  $x$
is $d$-step almost automorphic if and only if it is
$\F_{P_d}^*$-recurrent if and only if it is $\F_{B_d}^*$-recurrent.
\end{conj}

We note that there are two possible ways to show the conjecture: (1)
prove $\F_{P_d}$ and $\F_{B_d}$ have the Ramsey property, (2) prove
$\F_{P_d}\subset \F_{B_d}\subset \F^*_{d,0}$. Unfortunately, at this
moments we can not prove neither of them.

\subsection{Organization of the paper}

We organize the paper as follows: in Section 2, we give the basic
definitions and facts used in the paper. In Section 3, we study
Nil$_d$-Bohr$_0$ sets and higher order recurrence sets, and use them
to characterize $\RP^{[d]}$. In Section 4, we study $SG_d$ sets and
use them to characterize $\RP^{[d]}$. In Section 5, we introduce the
cubic version of multiple recurrence sets, and also use them to
characterize $\RP^{[d]}$. In the final section, we introduce the
notion of $d$-step almost automorphy and obtain various
characterizations. In the Appendix, we show $SG_2$ does not have the
Ramsey property, Theorem \ref{ShaoYe} holds for general compact
Hausdorff systems and the cubic version of the multiple Poincar\'e
and Birkhoff recurrence sets can be interpreted using
intersectiveness.

\section{Preliminaries}

\subsection{Measurable and topological dynamics}

In this subsection we give some basic notions in ergodic theory and
topological dynamics.

\subsubsection{Measurable systems}
In this paper, a {\em measure preserving system}  is a quadruple
$(X,\X, \mu, T)$, where $(X,\X,\mu )$ is a Lebesgue probability
space and $T : X \rightarrow X$ is an invertible measure preserving
transformation.

\medskip

We write $\I=\I (T)$ for the $\sigma$-algebra $\{A\in \X : T^{-1}A =
A\}$ of invariant sets. A system is {\em ergodic} if every
$T$-invariant set has measure either $0$ or $1$. $(X,\X, \mu, T)$ is
{\em weakly mixing} if the product system $(X\times X, \X\times \X,
\mu\times \mu, T\times T)$ is erdogic.

\subsubsection{Topological dynamical systems}

A {\em transformation} of a compact metric space X is a
homeomorphism of X to itself. A {\em topological dynamical system},
referred to more succinctly as just a t.d.s. or a {\em system}, is a
pair $(X, T)$, where $X$ is a compact metric space and $T : X
\rightarrow X$ is a transformation. We use $\rho (\cdot, \cdot)$ to
denote the metric on $X$.

A t.d.s. $(X, T)$ is {\em transitive} if $X$ is uncountable, and
there exists some point $x\in X$ whose orbit $\O(x,T)=\{T^nx: n\in
\Z\}$ is dense in $X$.
The system is {\em minimal} if the orbit of any
point is dense in $X$. This property is equivalent to saying that X
and the empty set are the only closed invariant sets in $X$.
A {\em factor} of a t.d.s. $(X, T)$ is another t.d.s. $(Y, S)$ such
that there exists a continuous and onto map $\phi: X \rightarrow Y$
satisfying $S\circ \phi = \phi\circ T$. In this case, $(X,T)$ is
called an {\em extension } of $(Y,S)$. The map $\phi$ is called a
{\em factor map}.

\subsubsection{} We also make use of a more general definition of a
measurable or topological system. That is, instead of just a single
transformation $T$, we consider commuting homeomorphisms $T_1,\ldots
, T_k$ of $X$ or a countable abelian group of transformations.

\subsection{Cubes and faces} In the following subsections, we will
introduce notions about cubes, faces and dynamical parallelepipeds.
For more details  see \cite{HK05, HKM, HM}.

\subsubsection{} Let $X$ be a set, let $d\ge 1$ be an integer, and write
$[d] = \{1, 2,\ldots , d\}$. We view $\{0, 1\}^d$ in one of two
ways, either as a sequence $\ep=\ep_1\ldots \ep_d$ of $0'$s and
$1'$s, or as a subset of $[d]$. A subset $\ep$ corresponds to the
sequence $(\ep_1,\ldots, \ep_d)\in \{0,1\}^d$ such that $i\in \ep$
if and only if $\ep_i = 1$ for $i\in [d]$. For example, ${\bf
0}=(0,0,\ldots,0)\in \{0,1\}^d$ is the same as $\emptyset \subset
[d]$.

Let $V_d=\{0,1\}^d=2^{[d]}$ and $V_d^*=V_d\setminus \{{\bf
0}\}=V_d\setminus \{\emptyset\}$. If ${\bf n} = (n_1,\ldots, n_d)\in
\Z^d$ and $\ep\in \{0,1\}^d$, we define
$${\bf n}\cdot \ep = \sum_{i=1}^d n_i\ep_i .$$
If we consider $\ep$ as $\ep\subset [d]$, then ${\bf n}\cdot \ep  =
\sum_{i\in \ep} n_i .$

\subsubsection{}

We denote $X^{2^d}$ by $X^{[d]}$. A point ${\bf x}\in X^{[d]}$ can
be written in one of two equivalent ways, depending on the context:
$${\bf x} = (x_\ep :\ep\in \{0,1\}^d )= (x_\ep : \ep\subset [d]). $$
Hence $x_\emptyset =x_{\bf 0}$ is the first coordinate of ${\bf x}$.
For example, points in $X^{[2]}$ are like
$$(x_{00},x_{10},x_{01},x_{11})=(x_{\emptyset}, x_{\{1\}},x_{\{2\}},x_{\{1,2\}}).$$

For $x \in X$, we write $x^{[d]} = (x, x,\ldots , x)\in  X^{[d]}$.
The diagonal of $X^{[d]}$ is $\D^{[d]} = \{x^{[d]}: x\in X\}$.
Usually, when $d=1$, denote the diagonal by $\D_X$ or $\D$ instead
of $\D^{[1]}$.

A point ${\bf x} \in X^{[d]}$ can be decomposed as ${\bf x} = ({\bf
x'},{\bf  x''})$ with ${\bf x}', {\bf x}''\in X^{[d-1]}$, where
${\bf x}' = (x_{\ep0} : \ep\in \{0,1\}^{d-1})$ and ${\bf x}''=
(x_{\ep1} : \ep\in \{0,1\}^{d-1})$. We can also isolate the first
coordinate, writing $X^{[d]}_* = X^{2^d-1}$ and then writing a point
${\bf x}\in X^{[d]}$ as ${\bf x} = (x_\emptyset, {\bf x}_*)$, where
${\bf x}_*= (x_\ep : \ep\neq \emptyset) \in X^{[d]}_*$.

\subsection{Dynamical parallelepipeds}

\begin{de}
Let $(X, T)$ be a t.d.s. and let $d\ge 1$ be an integer. We define
$\Q^{[d]}(X)$ to be the closure in $X^{[d]}$ of elements of the form
$$(T^{{\bf n}\cdot \ep}x=T^{n_1\ep_1+\ldots + n_d\ep_d}x: \ep=
(\ep_1,\ldots,\ep_d)\in\{0,1\}^d) ,$$ where ${\bf n} = (n_1,\ldots ,
n_d)\in \Z^d$ and $ x\in X$. When there is no ambiguity, we write
$\Q^{[d]}$ instead of $\Q^{[d]}(X)$. An element of $\Q^{[d]}(X)$ is
called a (dynamical) {\em parallelepiped of dimension $d$}.
\end{de}

As examples, $\Q^{[2]}$ is the closure in $X^{[2]}=X^4$ of the set
$$\{(x, T^mx, T^nx, T^{n+m}x) : x \in X, m, n\in \Z\}$$ and $\Q^{[3]}$
is the closure in $X^{[3]}=X^8$ of the set $$\{(x, T^mx, T^nx,
T^{m+n}x, T^px, T^{m+p}x, T^{n+p}x, T^{m+n+p}x) : x\in X, m, n, p\in
\Z\}.$$

\begin{de}
Let $\phi: X\rightarrow Y$ and $d\in \N$. Define $\phi^{[d]}:
X^{[d]}\rightarrow Y^{[d]}$ by $(\phi^{[d]}{\bf x})_\ep=\phi x_\ep$
for every ${\bf x}\in X^{[d]}$ and every $\ep\subset [d]$.

Let $(X, T)$ be a system and $d\ge 1$ be an integer. The {\em
diagonal transformation} of $X^{[d]}$ is the map $T^{[d]}$.
\end{de}

\begin{de}
{\em Face transformations} are defined inductively as follows: Let
$T^{[0]}=T$, $T^{[1]}_1=\id \times T$. If
$\{T^{[d-1]}_j\}_{j=1}^{d-1}$ is defined already, then set
$$T^{[d]}_j=T^{[d-1]}_j\times T^{[d-1]}_j, \ j\in \{1,2,\ldots, d-1\},$$
$$T^{[d]}_d=\id ^{[d-1]}\times T^{[d-1]}.$$
\end{de}


The {\em face group} of dimension $d$ is the group $\F^{[d]}(X)$ of
transformations of $X^{[d]}$ spanned by the face transformations.
The {\em parallelepiped group} of dimension $d$ is the group
$\G^{[d]}(X)$ spanned by the diagonal transformation and the face
transformations. We often write $\F^{[d]}$ and $\G^{[d]}$ instead of
$\F^{[d]}(X)$ and $\G^{[d]}(X)$, respectively. For $\G^{[d]}$ and
$\F^{[d]}$, we use similar notations to that used for $X^{[d]}$:
namely, an element of either of these groups is written as $S =
(S_\ep : \ep\in\{0,1\}^d)$. In particular, $\F^{[d]} =\{S\in
\G^{[d]}: S_\emptyset ={\rm id}\}$.

\medskip

For convenience, we denote the orbit closure of ${\bf x}\in X^{[d]}$
under $\F^{[d]}$ by $\overline{\F^{[d]}}({\bf x})$, instead of
$\overline{\O({\bf x}, \F^{[d]})}$.

It is easy to verify that $\Q^{[d]}$ is the closure in $X^{[d]}$ of
$$\{Sx^{[d]} : S\in \F^{[d]}, x\in X\}.$$
If $x$ is a transitive point of $X$, then $\Q^{[d]}$ is the closed
orbit of $x^{[d]}$ under the group $\G^{[d]}$.

\subsection{Nilmanifolds and nilsystems}




\subsubsection{Nilpotent groups}

Let $G$ be a group. For $g, h\in G$, we write $[g, h] =
ghg^{-1}h^{-1}$ for the commutator of $g$ and $h$ and we write
$[A,B]$ for the subgroup spanned by $\{[a, b] : a \in A, b\in B\}$.
The commutator subgroups $G_j$, $j\ge 1$, are defined inductively by
setting $G_1 = G$ and $G_{j+1} = [G_j ,G]$. Let $k \ge 1$ be an
integer. We say that $G$ is {\em $k$-step nilpotent} if $G_{k+1}$ is
the trivial subgroup.

\subsubsection{Nilmanifolds}

Let $G$ be a $k$-step nilpotent Lie group and $\Gamma$ a discrete
cocompact subgroup of $G$. The compact manifold $X = G/\Gamma$ is
called a {\em $k$-step nilmanifold}. The group $G$ acts on $X$ by
left translations and we write this action as $(g, x)\mapsto gx$.
The Haar measure $\mu$ of $X$ is the unique probability measure on
$X$ invariant under this action. Let $\tau\in G$ and $T$ be the
transformation $x\mapsto \tau x$ of $X$. Then $(X, T, \mu)$ is
called a {\em  $k$-step nilsystem}.

\subsubsection{$d$-step nilsystem and system of order $d$}

We also make use of inverse limits of nilsystems and so we recall
the definition of an inverse limit of systems (restricting ourselves
to the case of sequential inverse limits). If $(X_i,T_i)_{i\in \N}$
are systems with $diam(X_i)\le M<\infty$ and $\phi_i:
X_{i+1}\rightarrow X_i$ are factor maps, the {\em inverse limit} of
the systems is defined to be the compact subset of $\prod_{i\in
\N}X_i$ given by $\{ (x_i)_{i\in \N }: \phi_i(x_{i+1}) = x_i,
i\in\N\}$, which is denoted by $\displaystyle
\lim_{\longleftarrow}\{X_i\}_{i\in \N}$. It is a compact metric
space endowed with the distance $\rho(x, y) =\sum_{i\in \N} 1/2^i
\rho_i(x_i, y_i )$. We note that the maps $\{T_i\}$ induce a
transformation $T$ on the inverse limit.

\begin{de}[Host-Kra-Maass]\cite{HKM}\label{HKM}
A transitive t.d.s. $(X,T)$ is called a {\em system of order $d$},
if it is an inverse limit of $d$-step minimal nilsystems.
\end{de}

\subsection{Families and filters}

\subsubsection{Furstenberg families}

We say that a collection $\F$ of subsets of $\Z$ is a {\em a family}
if it is hereditary upward, i.e. $F_1 \subseteq F_2$ and $F_1 \in
\F$ imply $F_2 \in \F$. A family $\F$ is called {\em proper} if it
is neither empty nor the entire power set of $\Z$, or, equivalently
if $\Z \in \F$ and $\emptyset \not\in \F$. Any nonempty collection
$\mathcal{A}$ of subsets  of $\Z$ generates a family
$\F(\mathcal{A}) :=\{F \subseteq \Z:F \supset A$ for some $A \in
\mathcal{A}\}$.

For a family $\F$ its {\em dual} is the family $\F^{\ast}
:=\{F\subseteq \Z  : F \cap F' \neq \emptyset \ \text{for  all} \ F'
\in \F \}$. It is not hard to see that $\F^*=\{F\subset\Z:
\Z\setminus F\not\in \F\}$, from which we have that if $\F$ is a
family then $(\F^*)^*=\F.$ For more details, see \cite{Ak}.

\subsubsection{Filter and the Ramsey property}

If a family $\F$ is closed under finite intersections and is proper,
then it is called a {\em filter}.

A family $\F$ has {\em the Ramsey property} if $A=A_1\cup A_2\in \F$
then $A_1\in \F$ or $A_2\in \F$. It is well known that a proper
family has the Ramsey property if and only if its dual $\F^*$ is a
filter \cite{F}.

\subsubsection{Some important families}

A subset $S$ of $\Z$ is {\it syndetic} if it has a bounded gaps,
i.e. there is $N\in \N$ such that $\{i,i+1,\cdots,i+N\} \cap S \neq
\emptyset$ for every $i \in {\Z}$. The collection of all syndetic
subsets is denoted by $\F_s$.

Let $S$ be a subset of $\mathbb{Z}$. The {\it upper Banach density}
and {\it lower Banach density} of $S$ are
$$BD^*(S)=\limsup_{|I|\to \infty}\frac{|S\cap I|}{|I|},\ \text{and}\
BD_*=\liminf_{|I|\to \infty}\frac{|S\cap I|}{|I|},$$ where $I$
ranges over intervals of $\mathbb{Z}$, while the {\it upper density}
of $S$ is
$$D^*(S)=\limsup_{n\to \infty}\frac{|S\cap [-n,n]|}{2n+1}.$$


Let $\{ b_i \}_{i\in I}$ be a finite or infinite sequence in
$\mathbb{Z}$. One defines $$FS(\{ b_i \}_{i\in I})=\Big\{\sum_{i\in
\alpha} b_i: \alpha \text{ is a finite non-empty subset of } I\Big
\}.$$ $F$ is an {\it IP set} if it contains some
$FS({\{p_i\}_{i=1}^{\infty}})$, where $p_i\in\Z$. The collection of
all IP sets is denoted by $\F_{ip}$.
If $I$ is finite, then one says $FS(\{ p_i \}_{i\in I})$ is an {\em
finite IP set}. The collection of all sets containing finite IP sets
with arbitrarily long lengths is denoted by $\F_{fip}$.

\subsection{Regionally proximal pairs of order $d$}

First recall the definition of regionally proximal pairs of order
$d$. Let $(X, T)$ be a t.d.s. and let $d\ge 1$ be an integer. A pair
$(x, y) \in X\times X$ is said to be {\em regionally proximal of
order $d$} if for any $\d  > 0$, there exist $x', y'\in X$ and a
vector ${\bf n} = (n_1,\ldots , n_d)\in\Z^d$ such that $\rho (x, x')
< \d, \rho (y, y') <\d$, and
$$\rho(T^{{\bf n}\cdot \ep}x', T^{{\bf n}\cdot \ep}y') < \d\
\text{for any nonempty $\ep\subset [d]$}.$$ The set of regionally
proximal pairs of order $d$ is denoted by $\RP^{[d]}$ (or by
$\RP^{[d]}(X)$ in case of ambiguity), which is called {\em the
regionally proximal relation of order $d$}.

Moreover, let $\RP^{[\infty]}=\bigcap_{d=1}^\infty\RP^{[d]}(X)$. The
following theorem was proved by Host-Kra-Maass for minimal distal
systems \cite{HKM} and by Shao-Ye \cite{SY} for the general minimal
systems.

\begin{thm}\label{ShaoYe}
Let $(X, T)$ be a minimal t.d.s. and $d\in \N$. Then
\begin{enumerate}
\item $(x,y)\in \RP^{[d]}$ if and only if $(x,y,y,\ldots,y)=(x,y^{[d+1]}_*) \in
\overline{\F^{[d+1]}}(x^{[d+1]})$ if and only if $(x,x^{[d]}_*, y,
x^{[d]}_*) \in \overline{\F^{[d+1]}}(x^{[d+1]})$.

\item $(\overline{\F^{[d]}}(x^{[d]}),\F^{[d]})$ is minimal for all
$x\in X$.

\item $\RP^{[d]}(X)$ is an equivalence relation, and so is
$\RP^{[\infty]}.$

\item If $\pi:(X,T)\lra (Y,S)$ is a factor map, then $(\pi\times
\pi)(\RP^{[d]}(X))=\RP^{[d]}(Y).$

\item $(X/\RP^{[d]},T)$ is the maximal nilfactor of $(X,T)$.
\end{enumerate}\end{thm}
Note that (5) means that there is $d\in\N$ such that
$(X/\RP^{[d]},T)$ is a system of order $d$ and any system of order
$d$ factor of $(X,T)$ is a factor of $(X/\RP^{[d]},T)$.

\begin{rem}
In \cite{SY}, Theorem \ref{ShaoYe} was proved for compact metric
spaces. In fact, one can show that Theorem \ref{ShaoYe} holds for
compact Hausdorff spaces by repeating the proofs sentence by
sentence in \cite{SY}. However, we will describe a direct approach
in Appendix \ref{appendix:CT2}. This result will be used in the next
section.
\end{rem}

\section{Nil$_d$ Bohr$_0$-sets, Poincar\'e sets and $\RP^{[d]}$}

In this section using  results obtained in \cite{HSY} we
characterize $\RP^{[d]}$ using the families $\F_{Poi_d}, \F_{Bir_d}$
and $\F_{d,0}^*$.

\subsection{Nil-Bohr sets}

Recall $\F_{d,0}$ is the family generated by all Nil$_d$
Bohr$_0$-sets.


\medskip

For $F_1,F_2\in \F_{d,0}$, there are $d$-step nilsystems $(X,T)$,
$(Y,S)$, $(x,y)\in X\times Y$ and $U\times V$ neighborhood of
$(x,y)$ such that $N(x,U)\subset F_1$ and $N(y,V)\subset F_2$. It is
clear that $N(x,U)\cap N(y, V)=N((x,y), U\times V)\in \F_{d,0}$.
This implies that $F_1\cap F_2\in\F_{d,0}$. So we conclude that

\begin{prop}
Let $d \in \N$. Then $\F_{d,0}$ is a filter, and $\F_{d,0}^*$ has
the Ramsey property.
\end{prop}

\subsection{Sets of $d$-recurrence}

\subsubsection{}

Recall that for $d\in \N$, $\F_{Poi_d}$ (resp. $\F_{Bir_d}$) is the
family generated by the collection of all sets of $d$-recurrence
(resp. sets of $d$-topological recurrence).

\begin{rem} It is known that for all integer $d\ge 2$ there exists a set of
$(d-1)$-recurrence that is not a set of $d$-recurrence \cite{FLW}.
This also follows from Theorem~ \ref{intro-10}.
\end{rem}

Recall that a set $S\subseteq \Z$ is {\em $d$-intersective} if every
subset $A$ with positive density contains at least one arithmetic
progression of length $d+1$ and a common difference in $S$, i.e.
there is some $n\in S$ such that
$$A\cap (A-n)\cap (A-2n)\ldots\cap (A-dn)\neq \emptyset.$$
Similarly, one can define topological $d$-intersective set by
replacing the set with positive density by a syndetic set in the
above definition.

We now give some equivalence conditions of $d$-topological
recurrence.
\begin{prop}\label{birkhoff-equi1} The following statements are equivalent:
\begin{enumerate}
\item $S\subset \Z$ is a set of topological $d$-intersective.

\item $S\subset \Z$ is a set of {$d$-topological recurrence}.

\item For any t.d.s. $(X,T)$ there are $x\in X$ and
$\{n_i\}_{i=1}^\infty\subset S$ such that
$$\lim_{i\lra +\infty}T^{jn_i}x=x\ \text{for each}\ 1\le j\le d.$$
\end{enumerate}
\end{prop}
\begin{proof} The equivalence between (1) and (2) was proved in \cite{FLW,
F81}.

$(2)\Rightarrow(3)$. Now assume that whenever $(Y, S)$ is a minimal
t.d.s. and $V\subseteq Y$ a nonempty open set, then there is $n\in
S$ such that
$$V\cap T^{-n}V\cap \ldots\cap T^{-dn}V\not=\emptyset.$$
Let $(X, T)$ be a t.d.s., and without loss of generality we assume
that $(X,T)$ is minimal, since each t.d.s. contains a minimal
subsystem. Define for each $j\in\N$
$$W_j=\{x\in X: \exists\ n\in S\
\text{with}\ d(T^{kn}x,x)<\frac{1}{j}\ \text{for each}\ 1\le k\le
d\}.$$ Then it is easy to verify that $W_j$ is non-empty, open and
dense. Then any $x\in \bigcap_{j=1}^\infty W_j$ is the point we look
for.

$(3)\Rightarrow(2).$  Let $(X, T)$ be a minimal t.d.s. and
$U\subseteq X$ a nonempty open set. Then there are $x\in X$ and
$\{n_i\}_{i=1}^\infty\subset S$ such that for each given $1\le k\le
d$, $T^{kn_i}x\lra x$. Since $(X,T)$ is minimal, there is some $l\in
\Z$ such that $x\in V=T^{-l}U$. When $i_0$ is larger enough, we have
$V \cap T^{-n_{i_0}}V\cap \ldots \cap T^{-dn_{i_0}}V\neq \emptyset$,
which implies that $U\cap T^{-n}U\cap \ldots \cap T^{-dn}U\neq
\emptyset$ by putting $n=n_{i_0}$.

\end{proof}

\subsubsection{}

The following fact follows from the Poincar\'e and Birkhoff multiple
recurrent theorems.
\begin{prop}\label{ramseyp}
For all $d\in \N$, $\F_{Poi_d}$ and $\F_{Bir_d}$ have the Ramsey
property.
\end{prop}
\begin{proof} Let $F\in \F_{Poi_d}$ and $F=F_1\cup F_2$. Assume the
contrary that $F_i\not \in \F_{Poi_d}$ for $i=1,2$. Then there are
measure preserving systems $(X_i,\mathcal{B}_i,\mu_i,T_i)$ and
$A_i\in \mathcal{B}_i$ with $\mu_i(A_i)>0$ such that $\mu_i(A_i\cap
T_i^{-n}A_i\cap\ldots\cap T_i^{-dn}A_i)=0$ for $n\in F_i$, where
$i=1,2$. Set $\mu=\mu_1\times \mu_2$, $A=A_1\times A_2$ and
$T=T_1\times T_2$. Then we have
\begin{equation*}
\begin{split}
& \quad \ \mu(A\cap T^{-n}A\cap \ldots\cap T^{-dn}A)
       \\&=\mu_1(A_1\cap T_1^{-n}A_1\cap \ldots\cap T_1^{-dn}A_1)
       \mu_2(A_2\cap T_2^{-n}A_2\cap \ldots\cap T_2^{-dn}A_2)=0
\end{split}
\end{equation*}
for each $n\in F=F_1\cup F_2$, a contradiction. The other case can
be shown similarly.
\end{proof}
\subsection{Nil$_d$ Bohr$_0$-sets and $\RP^{[d]}$}

To show the following result we need several well known facts
(related to distality) from the Ellis enveloping semigroup theory,
see \cite{Au88, G76, V77, Vr}. Also we note that the lifting
property in Theorem \ref{ShaoYe} is valid when $X$ is compact and
Hausdorff (see Appendix \ref{appendix:CT2} for more details).

\begin{thm}\label{huang10}
Let $(X,T)$ be a minimal t.d.s.. Then $(x,y)\in \RP^{[d]}$ if and
only if $N(x,U)\in \F_{d,0}^*$ for each neighborhood $U$ of $y$.
\end{thm}

\begin{proof}
First assume that $N(x,U)\in \F_{d,0}^*$ for each neighborhood $U$
of $y$. Let $(X_d,S)$ be the maximal $d$-step nilfactor of $(X,T)$
(see Theorem \ref{ShaoYe}) and $\pi:X\lra X_d$ be the projection.
Then for any neighborhood $V$ of $\pi(x)$, we have $N(x,U)\cap
N(\pi(x),V)\not=\emptyset$ since $N(x,U)\in \F_{d,0}^*$. This means
that there is a sequence $\{n_i\}$ such that $$(T\times S)^{n_i}(x,
\pi(x))\lra (y,\pi(x)),\ i\to \infty .$$ Thus, we have $$
\pi(y)=\pi(\lim_{i}T^{n_i}x)=\lim_{i}S^{n_i}\pi(x)=\pi(x),$$ i.e.
$(x,y)\in \RP^{[d]}$.

\medskip
Now assume that $(x,y)\in \RP^{[d]}$ and $U$ is a neighborhood of
$y$. We need to show that if $(Z,R)$ is a $d$-step nilsystem,
$z_0\in Z$ and $V$ is a neighborhood of $z_0$ then $N(x,U)\cap
N(z_0,V)\not=\emptyset$.

Let $$W=\prod_{z\in Z}Z \quad \text{ (i.e. $W=Z^Z$) and
}R^Z:W\rightarrow W$$ with $(R^Z\omega)(z)=R(\omega(z))$ for any
$z\in Z$, where $\omega=(\omega(z))_{z\in Z}\in W$. Note that in
general $(W,R^Z)$ is not a metrizable but a compact Hausdorff
system. Since $(Z,R)$ is a $d$-step nilsystem, $(Z,R)$ is distal.
Hence $(W,R^Z)$ is also distal.

Choose $\omega^*\in W$ with $\omega^*(z)=z$ for $z\in Z$, and let
$Z_\infty=\text{cl}(\text{orb}(\omega^*,R^Z))$. Then
$(Z_\infty,R^Z)$ is a minimal subsystem of $(W,R^Z)$ since $(W,R^Z)$
is distal. For $\omega\in Z_\infty$, there exists $p\in E(Z,R)$ such
that $\omega(z)=p(\omega^*(z))=p(z)$ for $z\in Z$. Since $(Z,R)$ is
a minimal distal system, the Ellis semigroup $E(Z,R)$ is a group
(Appendix \ref{appendix:CT2}). Particularly, $p:Z\rightarrow Z$ is a
surjective map. Thus
$$\{\omega(z):z\in Z\}=\{ p(z):z\in Z\}=Z.$$ Hence there there
exists $z_\omega\in Z$ such that $\omega(z_\omega)=z_0$.

\medskip

Take a minimal subsystem $(A,T\times R^Z)$ of the product system
$(X\times Z_\infty,T\times R^Z)$. Let $\pi_X: A\rightarrow X$ be the
natural coordinate projection. Then $\pi_X:(A,T\times
R^Z)\rightarrow (X,T)$ is a factor map between two minimal systems.
Since $(x,y)\in \RP^{[d]}(X,T)$, by Theorem \ref{ShaoYe} there exist
$\omega^1,\omega^2\in W$ such that $((x,\omega^1),(y,\omega^2))\in
\RP^{[d]}(A,T\times R^Z)$.

For $\omega^1$, there exists $z_1\in Z$ such that
$\omega^1(z_1)=z_0$ by the above discussion. Let $\pi: A\rightarrow
X\times Z$ with $\pi(u,\omega)=(u,\omega(z_1))$ for $(u,\omega)\in
A$, $u\in X$, $\omega\in W$. Let $B=\pi(A)$. Then $(B,T\times R)$ is
a minimal subsystem of $(X\times Z,T\times R)$, and $\pi:(A,T\times
R^Z)\rightarrow (B,T\times R)$ is a factor map between two minimal
systems. Clearly $\pi(x,\omega^1)=(x,z_0)$,
$\pi(y,\omega^2)=(y,z_2)$ for some $z_2\in Z$, and
$$((x,z_0),(y,z_2))=\pi\times \pi((x,\omega^1),(y,\omega^2))\in
\RP^{[d]}(B,T\times R).$$ Moreover, we consider the projection
$\pi_Z$ of $B$ onto $Z$. Then $\pi_Z:(B,T\times R)\rightarrow (Z,R)$
is a factor map and so $(z_0,z_2)=\pi_Z\times
\pi_Z((x,z_0),(y,z_2))\in \RP^{[d]}(Z,R)$. Since $(Z,R)$ is a system
of order $d$, $z_0=z_2$. Thus $((x,z_0),(y,z_0))\in
\RP^{[d]}(B,T\times R)$. Particularly, $N(x,U)\cap
N(z_0,V)=N((x,z_0),U\times V)$ is a syndetic set since $(B,T\times
R)$ is minimal. This completes the proof of theorem.
\end{proof}

\begin{rem}\label{huangrem} From the proof of Theorem \ref{huang10},
we have the following result: Let $(X,T)$ be a minimal system and
$(x,y)\in \RP^{[d]}$. Then $N(x,U)\cap F$ is a syndetic set for each
$F\in \F_{d,0}$ and each neighborhood $U$ of $y$.
\end{rem}


\subsection{Sets of $d$-recurrence and nilsequences}

It is known that $d$-recurrence sets are ``almost'' $d$-step
nilsequences \cite{HSY}. This result stated in Theorem
\ref{d-recurrence} follows from Propositions \ref{bhk} and \ref{B-F}
by a discussion in \cite{HSY}.

\begin{prop}\cite[Theorem 1.9]{BHK05} \label{bhk} Let $(X,\X,\mu, T )$ be
an ergodic measure preserving system, let $f \in L^\infty(\mu)$ and
let $d\ge 1$ be an integer. The sequence $\{I_f(d,n)\}$ is the sum
of a sequence tending to zero in uniform density and a $d$-step
nilsequence, where
\begin{equation}\label{}
  I_f(d,n)=\int f(x)f(T^nx)\ldots f(T^{dn}x)\ d\mu(x).
\end{equation}

Especially, for any $A\in \X$
$$\{I_{1_A}(d,n)\}=\{\mu(A\cap T^{-n}A\cap \ldots \cap T^{-dn}A)\}=F_d+N,$$
where $F_d$ is a $d$-step nilsequence and $N$ tending to zero in
uniform density.\end{prop}

\begin{prop} \cite{FK} or \cite[Theorem 6.15]{BM20}\label{B-F} Let $(X,\X,\mu, T )$ be an ergodic
measure preserving system and $d\in \N$. Then for $A\in \X$ with
$\mu(A)>0$ there is $c>0$ such that
$$\{n\in \Z: \mu(A\cap T^{-n}A\cap \ldots \cap T^{-dn}A)>c\}$$ is an
$IP^*$-set.
\end{prop}

\begin{thm}\label{d-recurrence}
Let $(X,\X,\mu, T )$ be an ergodic  measure preserving system and
$d\in \N$. Then for all $A\in \X$ with $\mu(A)>0$ the set $$I=\{n\in
\Z: \mu(A\cap T^{-n}A\cap \ldots \cap T^{-dn}A)>0\}$$ is an
``almost'' Nil$_d$ Bohr$_0$-set, i.e. there is some subset $M$ with
$BD^*(M)=0$ such that $I \D M$ is a Nil$_d$ Bohr$_0$-set.
\end{thm}

As an immediate consequence, one has

\begin{cor}\label{cor-3.12}
Let $(X,T)$ be a minimal t.d.s. and $d\in \N$. If $(x,y)\in
\RP^{[d]}$, then $N(x,U)\in \F_{Poi_d}$ and $N(x,U)\in \F_{Bir_d}$
for each neighborhood $U$ of $y$.
\end{cor}
\begin{proof} Let $U$ be a neighborhood of $y$. We have shown in Theorem \ref{huang10}
that $(x,y)\in \RP^{[d]}$ if and only if $N(x,U)\in \F_{d,0}^*.$
This means that $N(x,U)\cap B\not=\emptyset$ for each
$B\in\F_{d,0}.$

Now let $(X,\X,\mu, T )$ be an ergodic measure preserving system and
$A\in \X$ with $\mu(A)>0$. Set
$$F=\{n\in\Z:\mu(A\cap T^{-n}A\cap \ldots\cap T^{-dn}A)>0\}.$$
By Theorem \ref{d-recurrence} there is some subset $M$ with
$BD^*(M)=0$ such that $B=F \D M$ is a Nil$_d$ Bohr$_0$-set. Hence we
have $N(x,U)\cap (F\Delta M)$ is syndetic by Remark \ref{huangrem}.
Thus we conclude that there is $n\not=0$ with $n\in N(x,U)\cap F$
since $BD^*(M)=0$. By the definition, $N(x,U)\in \F_{Poi_d}\subset
\F_{Bir_d}$. The proof is completed.
\end{proof}

\subsection{A result concerning Nil$_d$ Bohr$_0$-sets}

To show the converse of Corollary \ref{cor-3.12}, we need the
following result.

\begin{thm}\cite{HSY}\label{HSY2011}
Let $d\in\N$. Then
$$\F_{Poi_d}\subset\F_{Bir_d}\subset \F^*_{d,0} . $$
\end{thm}

\subsection{Recurrence sets and $\RP^{[d]}$}

Now we can sum up the main results of this section as follows:

\begin{thm}\label{several}
Let $(X,T)$ be a minimal t.d.s.. Then the following statements are
equivalent:

\begin{enumerate}

\item $(x,y)\in \RP^{[d]}$.

\item
$N(x,U)\in \F_{Poi_d}$ for each neighborhood $U$ of $y$.

\item $N(x,U)\in \F_{Bir_d}$ for each
neighborhood $U$ of $y$.

\item $N(x,U)\in \F_{d,0}^*$ for each
neighborhood $U$ of $y$.
\end{enumerate}
\end{thm}

\begin{proof}
By Corollary \ref{cor-3.12} one has that $(1)\Rightarrow (2)$. It
follows from Theorem \ref{HSY2011} that $(2)\Rightarrow (3)
\Rightarrow (4)$. By Theorem \ref{huang10}, one has that
$(4)\Rightarrow (1)$ and completes the proof.
\end{proof}

\section{$SG_d$-sets and $\RP^{[d]}$}
In this section we will describe $\RP^{[d]}$ using the $SG_d$-sets
introduced by Host and Kra in \cite{HK10}. First we recall some
definitions.

\subsection{Sets $SG_d(P)$}

\begin{de}
Let $d\ge 0$ be an integer and let $P=\{p_i\}_i$ be a (finite or
infinite) sequence in $\Z$. The {\em set of sums with gaps of length
less than $d$} of $P$ is the set $SG_d(P)$ of all integers of the
form $$\ep_1p_1+\ep_2p_2+\ldots +\ep_np_n$$ where $n\ge 1$ is an
integer, $\ep_i\in \{0,1\}$ for $1\le i\le n$, the $\ep_i$ are not
all equal to $0$, and the blocks of consecutive $0$'s between two
$1$ have length less than $d$.

A subset $A \subseteq \Z$ is an $SG^*_d$-set if $A \cap SG_d(P)\neq
\emptyset$ for every infinite sequence $P$ in $\Z$.
\end{de}

Note that in this definition, $P$ is a sequence and not a subset of
$\Z$. For example, if $P=\{p_1,p_2,\ldots\}$, then $SG_1(P)$ is the
set of all sums $p_m+p_{m+1}+\ldots +p_n$ of consecutive elements of
$P$, and thus it coincides with the set $\D(S)$ where $S=\{p_1,
p_1+p_2, p_1+p_2+p_3,\ldots\}$. Therefore $SG^*_1$-sets are the same
as $\D^*$-sets.

For a sequence $P$, $SG_2(P)$ consists of all sums of the form
$$\sum_{i=m_0}^{m_1}p_i +\sum_{i=m_1+2}^{m_2}p_i+\ldots+
\sum_{i=m_{k-1}+2}^{m_k}p_i+\sum_{i=m_k+2}^{m_{k+1}}p_i$$ where
$k\in \N$ and $m_0,m_1,\ldots, m_{k+1}$ are positive integers
satisfying $m_{i+1}\ge m_i+2$ for $i=0,1,\ldots,k$.

\medskip

Denote by $SG_d$ the collection of all sets $SG_d(P)$ with $P$
infinite, and $\F_{SG_d}$ the family generated by $SG_d$ for each
$d\in \N$. Moreover, let $\F_{fSG_d}$ be the family containing
arbitrarily long $SG_d(P)$ sets with $P$ finite. That is, $A\in
\F_{fSG_d}$ if and only if there are finite sets $P^i$ with
$|P^i|\lra \infty$ such that $\bigcup_{i=1}^\infty SG_d(P^i)\subset
A$. It is clear that
$$\F_{SG_1}\supset \F_{SG_2}\supset \ldots\supset
\F_{SG_\infty}=:\bigcap_{i=1}^\infty
\F_{SG_i},$$ and
$$\F_{fSG_1}\supset \F_{fSG_2}\supset \ldots\supset
\F_{fSG_\infty}=:\bigcap_{i=1}^\infty \F_{fSG_i}.$$

We now show
\begin{prop} The following statements hold:
\begin{enumerate}
\item $\F_{SG_\infty}=\{A:\exists \ P^i\ \text{infinite for each}\ i\in
\N\ \text{such that}\ A\supset \bigcup_{i=1}^\infty SG_i(P^i)\}.$

\item $\F_{fSG_\infty}=\F_{fip}$.
\end{enumerate}
\end{prop}
\begin{proof} (1). Assume that $A\in \F_{SG_\infty}$. Then $A\in\bigcap_{i=1}^\infty
\F_{SG_i}$ and hence $A\in \F_{SG_i}$ for each $i\in \N$. Thus for
each $i\in\N$ there is $P^i$ infinite such that $A\supset SG_i(P^i)$
which implies that $A\supset\bigcup_{i=1}^\infty SG_i(P^i).$

Now let $B=\bigcup_{i=1}^\infty SG_i(P^i),$ where $P^i$ infinite for
each $i\in\N$. It is clear that $B\subset \F_{SG_i}$ for each $i$
and thus, $B\in\F_{SG_\infty}$. Since $\F_{SG_\infty}$ is a family,
we conclude that $\{A:\exists P^i\ \text{infinite for each}\ i\in
\N\ \text{such that}\ A\supset \bigcup_{i=1}^\infty
SG_i(P^i)\}\subset \F_{SG_\infty}$.

\medskip

(2) It is clear that $\F_{fSG_\infty}\subset \F_{fip}$. Let $A\in
\F_{fip}$ and without loss of generality assume that
$A=\bigcup_{i=1}^\infty FS (P^i)$ with $P^i=\{p_1^i, \ldots,
p_i^i\}$ and $|P^i|\lra\infty$.

Put $A_d=\bigcup_{i=1}^\infty SG_d(P^i)\subset A$ for $d\in\N$. Then
$A_d\in \F_{fSG_d}$ which implies that $A\in \F_{fSG_d}$ for each
$d\ge 1$ and hence $A\in\F_{fSG_\infty}$. That is, $\F_{fip}\subset
\F_{fSG_\infty}$.
\end{proof}

\subsection{$SG_d$-sets and $\RP^{[d]}$}

The following theorem is the main result of this section.

\begin{thm}\label{huang12}
Let $(X,T)$ be a minimal t.d.s.. Then for any $d\in\N$,
$(x,y)\in\RP^{[d]}$ if and only if $N(x,U)\in \F_{SG_d}$ for each
neighborhood $U$ of $y$. The same holds when $d=\infty$.
\end{thm}
\begin{proof}
It is clear that if $N(x,U)\in \F_{SG_d}$ for each neighborhood $U$
of $y$, then it contains some $FS(\{n_i\}_{i=1}^{d+1})$ for each
neighborhood $U$ of $y$ which implies that $(x,y)\in\RP^{[d]}$ by
Theorem \ref{ShaoYe}.

\medskip
Now assume that $(x,y)\in\RP^{[d]}$ for $d\ge 1$. Let for $i\ge 2$
$$A_i=:\{0,1\}^i\setminus \{(0,\ldots, 0,0), (0,\ldots,0, 1)\}$$

The case when $d=1$ was proved by Veech \cite{V68} and our method is
also valid for this case. To make the idea of the proof clearer, we
first show the case when $d=2$ and the general case follows by the
same idea.

\medskip

\noindent {\bf I. The case $d=2$.}

Assume that $(x,y)\in\RP^{[2]}$. Then by Theorem \ref{ShaoYe} (1)
and (2) for each neighborhood $V\times U$ of $(x,y)$, there are
$n_1,n_2, n_3\in \mathbb{N}$ such that
$$T^{\ep_1n_1+\ep_2n_2+\ep_3n_3}x\in V \ \text{and}\ T^{n_3}x\in U,$$
where $(\ep_1,\ep_2,\ep_3)\in A_3.$ For a given $U$, let $\eta>0$
with $B(y,\eta)\subset U$, and take $\eta_i>0$ with
$\sum_{i=1}^\infty\eta_i<\eta$, where $B(y,\eta)=\{x\in X:
\rho(x,y)<\eta\}$.

Choose $n_1^1,n_2^1,n_3^1\in \mathbb{N}$ such that
$$\rho(T^{n_3^1}x,y)<\eta_1\ \text{and}\ \rho(T^{r}x,x)<\eta_1,$$
where $r\in E_1$ with
$$E_1=\{\ep_1n_1^1+\ep_2n_2^1+\ep_3n_3^1: (\ep_1,\ep_2,\ep_3)\in
A_3\}.$$ Let
$$S_1=FS(\{n_1^1,n_2^1,n_3^1\}).$$

Choose $n_1^2,n_2^2,n_3^2\in \mathbb{N}$ such that
$$\rho(T^{n_3^2}x,y)<\eta_2 \text{ and } \max_{s\in
S_1}\rho(T^{s+r}x, T^sx)<\eta_2$$ for each $r\in E_2$ with
$$E_2=\{\ep_1n_1^2+\ep_2n_2^2+\ep_3n_3^2:
(\ep_1,\ep_2,\ep_3)\in A_3\}.$$ Let
$$S_2=FS(\{n_i^j:j=1,2,i=1,2,3\}).$$

Generally when $n_1^i,n_2^i,n_3^i$, $E_i,S_i$ are defined for $1\le
i\le k$ choose $n_1^{k+1},n_2^{k+1},n_3^{k+1}\in \mathbb{N}$ such
that
\begin{equation}\label{huang1}
\rho(T^{n_3^{k+1}}x,y)<\eta_{k+1} \text{ and } \max_{s\in
S_k}\rho(T^{s+r}x, T^sx)<\eta_{k+1}. \end{equation} for each $r\in
E_{k+1}$, where
$$E_{k+1}=\{\ep_1n_1^{k+1}+\ep_2n_{2}^{k+1}+\ep_3n_3^{k+1}:
(\ep_1,\ep_2,\ep_3)\in A_3\}.$$ Let
$$S_{k+1}=FS(\{n_i^j:i=1,2,3, 1\le j\le k+1\}).$$

Now we define a subsequence $P=\{P_k\}$ such that
$$P_1=n_3^1+n_1^2+n_1^3, P_2=n_3^2+n_2^3+n_2^4, P_3=n_3^3+n_1^4+n_1^5,P_4=n_3^4+n_2^5+n_2^6, \cdots$$
That is, \begin{equation*}
    P_k=n^k_{3}+n^{k+1}_{k\ ({\rm mod}\ 2)}+n^{k+2}_{k\  ({\rm
    mod}\ 2)},
\end{equation*}
where we assume $2m\ ({\rm mod}\ 2)=2$ for $m\in \mathbb{N}$.  We
claim that $N(x,U)\supset SG_2(P).$

Let $n\in SG_2(P)$ then $n=\sum_{j=1}^k P_{i_j},$ where $1\le
i_{j+1}-i_j\le 2$ for $1\le j\le k-1$. By induction for $k$, it is
not hard to show that $n$ can be written as
$$n=a_1+a_2+\cdots+a_{i_k-i_1+3}$$ such that $a_1=n_3^{i_1}$, $a_j\in
E_{j+i_1-1}$ for $j=2,3,\cdots,i_k-i_1+1$ and $a_{i_k-i_1+2}\in \{
n_1^{i_k+1},n_2^{i_k+1},n_1^{i_k+1}+n_2^{i_k+1}\}$,
$a_{i_k-i_1+3}=n_{i_k \ ({\rm mod}\ 2)}^{i_k+2}$. In other words,
$n$ can be written as $n=a_1+a_2+\ldots+a_{i_k-i_1+3}$ with
$a_1=n_3^{i_1}$ and $a_j\in E_{i_1+j-1}$ for $2\le j\le i_k-i_1+3$.

Note that $\sum \limits_{\ell=1}^j a_\ell\in S_{i_1+\ell-1}$ and
$a_{j+1}\in E_{i_1+j}$ for $1\le j\le i_k-i_1+2$. Thus by
$(\ref{huang1})$ we have
$$\rho(T^{\sum_{i=1}^ja_i}x,T^{\sum_{i=1}^{j+1}a_i}x)<\eta_{j+i_1}$$
for $1\le j\le i_k-i_1+2$. This implies that
\begin{align*}
\rho(T^nx,y)&\le \rho(T^{\sum_{j=1}^{i_k-i_1+3}a_i}x,
T^{\sum_{j=1}^{i_k-i_1+2}
a_i}x)+\cdots+\rho(T^{n_3^{i_1}+a_2}x,T^{n_3^{i_1}}x)+
\rho(T^{n_3^{i_1}}x,y)\\
&<\sum_{j=0}^{i_k-i_1+2}\eta_{j+i_1}<\eta.
\end{align*}
That is, $n\in N(x,U)$ and hence $N(x,U)\supset SG_2(P).$

\medskip

\noindent {\bf II. The general case.}

Generally assume that $(x,y)\in\RP^{[d]}$ with $d\ge 2$. Then by
Theorem \ref{ShaoYe} (1) and (2) for each neighborhood $V\times U$
of $(x,y)$, there are $n_1,n_2,\cdots, n_{d+1}\in \mathbb{N}$ such
that
$$T^{\ep_1n_1+\ep_2n_2+\cdots+\ep_{d+1}n_{d+1}}x\in V \ \text{and}\ T^{n_{d+1}}x\in U,$$
where $(\ep_1,\ep_2,\cdots, \ep_{d+1})\in A_{d+1}.$ For a given $U$,
let $\eta>0$ with $B(y,\eta)\subset U$, and take $\eta_i>0$ with
$\sum_{i=1}^\infty\eta_i<\eta$.

Choose $n_1^1,n_2^1,\cdots,n_{d+1}^1\in \mathbb{N}$ such that
$\rho(T^{n_{d+1}^1}x,y)<\eta_1\ \text{and}\ \rho(T^{r}x,x)<\eta_1$
where $r\in E_1$ with
$$E_1=\{\ep_1n_1^1+\ep_2n_2^1+\cdots+\ep_{d+1}n_{d+1}^1: (\ep_1,\ep_2,\cdots,\ep_{d+1})\in
A_{d+1}\}.$$ Let
$$S_1=FS(\{n^1_1,\cdots,n^1_{d+1}\}).$$

Choose $n_1^2,n_2^2,\cdots,n_{d+1}^2$ such that
$$\rho(T^{n_{d+1}^2}x,y)<\eta_2, \text{and}\ \max_{s\in
S_1}\rho(T^{s+r}x, T^sx)<\eta_2$$ for each $r\in E_2$ with
$$E_2=\{\ep_1n_1^2+\ep_2n_2^2+\cdots+\ep_{d+1}n_{d+1}^2:
(\ep_1,\ep_2,\cdots,\ep_{d+1})\in A_{d+1}\}.$$ Let
$$S_2=FS(\{n_1^1,\cdots, n_{d+1}^1,n_1^2,\cdots,n_{d+1}^2\}).$$

Generally when $n_1^i,\ldots,n_{d+1}^i$, $E_i,S_i$ are defined for
$1\le i\le k$ choose $n_1^{k+1},\cdots,n_{d+1}^{k+1}\in \mathbb{N}$
such that
\begin{equation}\label{huang2}
\rho(T^{n_{d+1}^{k+1}}x,y)<\eta_{k+1}, \text{and}\ \max_{s\in
S_k}\rho(T^{s+r}x, T^sx)<\eta_{k+1}. \end{equation} for each $r\in
E_{k+1}$, where
$$E_{k+1}=\{\ep_1n_1^{k+1}+\ep_2n_{2}^{k+1}+\ldots+\ep_{d+1}n_{d+1}^{k+1}:
(\ep_1,\ep_2,\ldots,\ep_{d+1})\in A_{d+1}\}.$$ Let
$$S_{k+1}=FS(\{n_i^j:i=1,\ldots, d+1, 1\le j\le k+1\}).$$

Now we define a subsequence $P=\{P_k\}$ such that
\begin{eqnarray*}
P_1&=&n_{d+1}^1+n_1^2+\cdots+n_1^{d+1},
P_2=n_{d+1}^2+n_2^3+\cdots+n_2^{d+2},
\cdots,\\ P_d&=&n_{d+1}^d+n_d^{d+1}+\cdots+n_d^{2d},\\
P_{d+1}&=& n_{d+1}^{d+1}+n_1^{d+2}+\cdots+n_1^{2d+1},
P_{d+2}=n_{d+1}^{d+2}+n_2^{d+3}+\cdots+n_2^{2d+2},\cdots,\\
P_{2d}&=&n_{d+1}^{2d}+n_d^{2d+1}+\cdots+n_d^{3d}, \cdots
\end{eqnarray*}
That is,
\begin{equation*}
    P_k=n^k_{d+1}+n^{k+1}_{k\ ({\rm mod}\ d)}+\cdots+n^{k+d}_{k\  ({\rm
    mod}\
    d)},
\end{equation*}
where we assume $dm\ ({\rm mod}\ d)=d$ for $m\in \mathbb{N}$.

We claim that $N(x,U)\supset SG_d(P).$ Let $n\in SG_d(P)$ then
$n=\sum_{j=1}^k P_{i_j},$ where $1\le i_{j+1}-i_j\le d$ for $1\le
j\le k-1$. By induction for $k$, it is not hard to show that $n$ can
be written as
$$n=a_1+a_2+\cdots+a_{i_k-i_1+d+1}$$ such that $a_1=n_{d+1}^{i_1}$, $a_j\in
E_{j+i_1-1}$ for $j=2,3,\cdots,i_k-i_1+1$ and
$$a_{i_k-i_1+1+r}\in FS(\Big\{ n^{i_k+r}_\ell:  \ell \in \{ 1,2,\cdots,d\} \setminus \bigcup_{j=1}^{r-1} \{ i_k+j \ ({\rm mod}\ d)\}\Big\})$$
for $1\le r \le d$. In other words, $n$ can be written as
$n=a_1+a_2+\ldots+a_{i_k-i_1+d+1}$ with $a_1=n_{d+1}^{i_1}$ and
$a_j\in E_{i_1+j-1}$ for $2\le j\le i_k-i_1+d+1$.

 Note that $\sum_{\ell=1}^j a_\ell\in
S_{i_1+\ell-1}$ and  $a_{j+1}\in E_{i_1+j}$ for $1\le j\le
i_k-i_1+d$. Thus by $(\ref{huang2})$ we have
$$\rho(T^{\sum_{i=1}^ja_i}x,T^{\sum_{i=1}^{j+1}a_i}x)<\eta_{i_1+j}$$
for $1\le j\le i_k-i_1+d$. This implies that
\begin{align*}
\rho(T^nx,y)&\le \rho(T^{\sum_{j=1}^{i_k-i_1+d+1}a_i}x,
T^{\sum_{j=1}^{i_k-i_1+d}
a_i}x)+\cdots+ \rho(T^{n_{d+1}^{i_1}}x,y)\\
&<\sum_{j=0}^{i_k-i_1+d}\eta_{j+i_1}<\eta.
\end{align*}
That is, $n\in N(x,U)$ and hence $N(x,U)\supset SG_d(P)$ which
implies that $N(x,U)\in \F_{SG_d}$. The proof is completed.
\end{proof}



\section{Cubic version of multiple recurrence sets and $\RP^{[d]}$}

Cubic version of multiple ergodic averages was studied in
\cite{HK05}, and also was proved very useful in some other questions
\cite{HK09, HK10, HKM}.

In this section we will discuss the question how to describe
$\RP^{[d]}$ using cubic version of multiple recurrence sets. Since
by Theorem \ref{ShaoYe} one can use dynamical parallelepipeds to
characterize $\RP^{[d]}$, it seems natural to describe $\RP^{[d]}$
using the cubic version of multiple recurrence sets.

\subsection{Cubic version of multiple Birkhoff recurrence sets}

First we give definitions for the cubic version of multiple
recurrence sets. We leave  the equivalent statements in viewpoint of
intersective sets in Appendix \ref{appendix:Intersective}.

\subsubsection{Birkhoff recurrence sets}
First we recall the classical definition. Let $P\subset \Z$. $P$ is
called a {\em Birkhoff recurrence set} (or a {\em set of topological
recurrence}) if whenever $(X, T)$ is a minimal t.d.s. and
$U\subseteq X$ a nonempty open set, then $P\cap N(U,U)\neq
\emptyset$. Let $\F_{Bir}$ denote the collection of Birkhoff
recurrence subsets of $\Z$. An alternative definition is that for
any t.d.s. $(X,T)$ there are $\{n_i\}\subset P$ and $x\in X$ such
that $T^{n_i}x\lra x$. Now we generalize the above definition to the
higher dimension.

\begin{de}\label{def-cubic}
Let $P\subset \Z$ and $d\in \N$. $P$ is called a {\em Birkhoff
recurrence set of order $d$} (or a {\em set of topological
recurrence of order $d$}) if whenever $(X, T)$ is a t.d.s. there are
$x\in X$ and $\{n_i^j\}_{j=1}^d\subset P$, $i\in \N$, such that
$FS(\{n_i^j\}_{j=1}^d)\subset P, i\in \N$ and for each given
$\ep=(\ep_1,\ldots,\ep_d)\in \{0,1\}^d$, $T^{m_i}x\lra x$, where
$m_i=\ep_1n_i^1+\ldots+\ep_dn_i^d$, $i\in\N$. A subset $F$ of $\Z$
is a {\em Birkhoff recurrence set of order $\infty$} if it is a
Birkhoff recurrence set of order $d$ for any $d\ge 1$.

For example, when $d=2$ this means that there are sequence $\{n_i\},
\{m_i\}\subset P$ and $x\in X$ such that $\{n_i+m_i\}\subset P$ and
$T^{n_i}x\lra x, T^{m_i}x\lra x$, $T^{n_i+m_i}x\lra x$.
\end{de}

Similarly we can define (topologically) intersective of order $d$
and intersective of order $d$ (see Appendix
\ref{appendix:Intersective}). We have
\begin{prop}\label{birkhoff-equi2} Let $d\in \N$. The following statements are equivalent:
\begin{enumerate}
\item $P$ is a {Birkhoff recurrence set} of order $d$.

\item Whenever $(X, T)$ is a minimal t.d.s. and
$U\subseteq X$ a nonempty open set, then there are $n_1,\ldots,n_d$
with $FS(\{n_i\}_{i=1}^d)\subset P$ such that
$$U \cap \big(\bigcap_{n\in FS(\{n_i\}_{i=1}^d)}T^{-n}U \big)\neq \emptyset.$$

\item $P$ is (topologically) intersective of order $d$.
\end{enumerate}
\end{prop}

\begin{proof} $(1)\Leftrightarrow (2)$ follows from the proof of Proposition
\ref{birkhoff-equi1}. See Appendix \ref{appendix:Intersective} for
the proof $(1)\Leftrightarrow (3).$
\end{proof}

\begin{rem}
From the above proof, one can see that for a minimal t.d.s. the set
of recurrent point in the Definition \ref{def-cubic} is residual.
\end{rem}

\subsubsection{Some properties of Birkhoff sequences of order $d$ }
The family generated by the collection of all Birkhoff recurrence
sets of order $d$ is denoted by $\F_{B_d}$. We have
$$\F_{B_1}\supset \F_{B_2}\supset \ldots \supset \F_{B_d}\supset \ldots\supset
\F_{B_\infty}=:\bigcap_{d=1}^\infty \F_{B_d}.$$

We will show later (after Proposition \ref{longproof}) that
\begin{prop}\label{birkhoff} $\F_{B_\infty}=\F_{fip}.$
\end{prop}

\subsection{Birkhoff recurrence sets and $\RP^{[d]}$}

We have the following theorem

\begin{thm}\label{birkhoff}
Let $(X,T)$ be a minimal t.d.s.. Then for any
$d\in\N\cup\{\infty\}$, $(x,y)\in\RP^{[d]}$ if and only if
$N(x,U)\in \F_{B_d}$ for each neighborhood $U$ of $y$.
\end{thm}

\begin{proof} We first show the case when $d\in\N$. ($\Leftarrow$)  Let $d\in \N$ and assume $N(x,U)\in \F_{B_d}.$
Then there are $FS(\{n_i\}_{i=1}^d)\subset N(x,U)$ such that $U\cap
\bigcap_{n\in FS(\{n_i\}_{i=1}^d)}T^{-n}U\neq \emptyset.$ This means
that there is $y'\in U$ such that $T^ny'\in U$ for any $n\in
FS(\{n_i\}_{i=1}^d)$. Since $T^nx\in U$ for any $n\in
FS(\{n_i\}_{i=1}^d)$, we conclude that $(x,y)\in\RP^{[d]}$ by the
definition.

\medskip

($\Rightarrow$) Assume that $(x,y)\in\RP^{[d]}$ and $U$ is a
neighborhood of $y$. Let $(Z, R)$ be a minimal  t.d.s., $V$ be a
non-empty open subset of $Z$ and $\Lambda\subset X\times Z$ be a
minimal subsystem. Let $\pi:\Lambda\lra X$ be the projection. Since
$(x,y)\in\RP^{[d]}$ there are $z_1,z_2\in Z$ such that
$((x,z_1),(y,z_2))\in \RP^{[d]}(\Lambda, T\times R)$ by Theorem
\ref{ShaoYe}. Let $m\in\N$ such that $T^{-m}V$ be a neighborhood of
$z_2$. Then $U\times T^{-m}V$ is a neighborhood of $(y,z_2)$. By
Theorem \ref{ShaoYe}, there are $n_1,\ldots,n_{d+1}$ such that
$$N((x,z_1),U\times T^{-m}V)\supset FS(\{n_i\}_{i=1}^{d+1}).$$ This
implies that $\bigcap_{n\in FS(\{n_i\}_{i=1}^{d+1})}
T^{-n-m}V\not=\emptyset.$ Thus, $V\cap \bigcap_{n\in
FS(\{n_i\}_{i=1}^{d})} T^{-n}V\not=\emptyset,$ i.e. $N(x,U)\in
\F_{B_d}$.

The case $d=\infty$ is followed from the result for $d\in \N$ and
the definitions.
\end{proof}

\subsection{Cubic version of multiple Poincar\'e recurrence sets}

\subsubsection{Poincar\'e recurrence sets}

Now we give the cubic version of multiple Poincar\'e recurrence
sets.

\begin{de}
For $d\in \N$, a subset $F$ of $\Z$ is a {\em Poincar\'e sequence of
order $d$} if for each $(X,\mathcal{B},\mu,T)$ and $A\in
\mathcal{B}$ with positive measure there are $n_1,\ldots,n_d\in\Z$
such that $FS(\{n_i\}_{i=1}^d)\subset F$ and
$$\mu(A\cap\big (\bigcap _{n\in
FS(\{n_i\}_{i=1}^d)} T^{-n}A\big ))>0.$$

A subset $F$ of $\Z$ is a {\em Poincar\'e sequence of order
$\infty$} if it is a Poincar\'e sequence of order $d$ for any $d\ge
1$.
\end{de}

\begin{rem}
We remark that $F$ is a Poincar\'e sequence of order $1$ iff it is a
Poincar\'e sequence. Moreover, a Poincar\'e sequence of order $1$
does not imply that it is a Poincar\'e sequence of order $2$. For
example, $\{n^k: n\in\mathbb{N}\}$ ($k\ge 3$) is a Poincar\'e
sequence \cite{F}, it is not a Poincar\'e sequence of order $2$ by
the famous Fermat Last Theorem.
\end{rem}

\subsubsection{Some properties of Poincar\'e sequences of order $d$ }

Let for $d\in \N\cup\{\infty\}$, $\F_{P_d}$ be the family generated
by the collection of all Poincar\'e sequences of order $d$. Thus
$$\F_{P_1}=\F_{Poi}\supset \F_{P_2}\supset \ldots \supset
\F_{P_d}\supset \ldots\supset \F_{P_\infty}=:\bigcap_{d=1}^\infty
\F_{P_d}.$$

We want to show that $\F_{P_\infty}=\F_{fip}$. It is clear that
$\F_{P_\infty}\subset\F_{fip}$. To show $\F_{P_d}\supset\F_{fip},$
we need the following proposition, for a proof see \cite{G} or
\cite{HLY}.

\begin{prop}\label{infinite-type}
Let $(X,\mathcal{B},\mu)$ be a
probability space, and $\{E_i\}_{i=1}^\infty$ be a sequence of
measurable sets with $\mu(E_i)\ge a>0$ for some constant $a$ and any
$i\in\N$. Then for any $k\ge 1$ and $\ep>0$ there is $N=N(a, k,
\ep)$ such that for any tuple $\{ s_1<s_2<\cdots <s_n\}$ with $n\ge
N$ there exist $1\le t_1<t_2<\cdots<t_k\le n$ with
\begin{align}\label{bds-key}
\mu(E_{s_{t_1}}\cap E_{s_{t_2}}\cap \cdots \cap E_{s_{t_{k}}})\ge
a^k-\ep.
\end{align}
\end{prop}

\medskip

\begin{rem}\label{remark}
To prove Proposition \ref{longproof}, one needs to use Proposition
\ref{infinite-type} repeatedly. To avoid explaining the same idea
frequently, we illustrate how we will use Proposition
\ref{infinite-type} in the proof of Proposition \ref{longproof}
first.

\medskip

Let $\{k_i^j\}_{i=1}^\infty$ be subsequences of $\Z$, $j\in\N$.
Assume $(X,\mathcal{B},\mu,T)$ is a measure preserving system and
$A\in \mathcal{B}$ with positive measure. Let $A_1=A, a_1=\mu(A_1)$.
We will show that there are $A_j\in \mathcal{B}$ and $t_1^j,t_2^j,
N_j$ such that $a_j=\mu(A_j)\ge \frac 12a_{j-1}^2>0$, and for $n\ge
N_j$ and any tuple $\{ s_1<s_2<\cdots <s_n\}$ there exist $1\le
t_1^j<t_2^j\le n$ with $\mu(T^{-k^j_{s_{t_1^j}}}A_j\cap
T^{-k^j_{s_{t_2^j}}}A_j)\ge \frac{1}{2}a_j^2$.

Let $E^1_i=T^{-k^1_i}A, i\in \N$. Let $A_1=A, a_1=\mu(A_1)$ and let
$N_1=N(a_1,2,\frac{1}{2}a_1^2)$ be as in Proposition
\ref{infinite-type}. Then for $n\ge N_1$ and any tuple $\{
s_1<s_2<\cdots <s_n\}$ there exist $1\le t_1^1<t_2^1\le n$ with
$\mu(E^1_{s_{t_1^1}}\cap E^1_{s_{t_2^1}})\ge \frac{1}{2}a_1^2.$

Once one fixes a tuple $\{ s_1<s_2<\cdots <s_n\}$, then one has a
fixed $k^1_{t^1_1}$ and $k^1_{t^1_2}$ with
$\mu(E^1_{k^1_{t_1^1}}\cap E^1_{k^1_{t_2^1}})\ge \frac{1}{2}a_1^2$.
Now let $A_2=A_1\cap T^{-k^1_{t_2^1}+k^1_{t_1^1} }A_1$,
$a_2=\mu(A_2)=\mu(E^1_{k_{t_1^1}}\cap E^1_{k_{t_2^1}})\ge \frac 12
a_1^2$. Let $E^2_i=T^{-k^2_i}A_2, i\in \N$. Let
$N_2=N(a_2,2,\frac{1}{2}a_2^2)$ be as in Proposition
\ref{infinite-type}.  Thus for $n\ge N_2$ and any tuple $\{
s_1<s_2<\cdots <s_n\}$ there exist $1\le t_1^2<t_2^2\le n$ with
$\mu(E^2_{s_{t_1^2}}\cap E^2_{s_{t_2^2}})\ge \frac{1}{2}a_2^2.$ Then
one fixes a tuple $\{ s_1<s_2<\cdots <s_n\}$ and goes on as above.

Inductively, assume that $\{E^j_i=T^{-k^{j}_i}A_{j}\}_{i=1}^\infty,
A_j, a_j, t_1^j,t_2^j, N_j$ are defined such that for $n\ge N_j$ and
any tuple $\{ s_1<s_2<\cdots <s_n\}$ there exist $1\le
t_1^j<t_2^j\le n$ with $\mu(E^2_{s_{t_1^j}}\cap E^2_{s_{t_2^j}})\ge
\frac{1}{2}a_j^2$. Fix a tuple $\{ s_1<s_2<\cdots <s_n\}$, then one
has a fixed $k^j_{t^1_1}$ and $k^j_{t^1_2}$ with
$\mu(E^j_{k^j_{t_1^j}}\cap E^j_{k^j_{t_2^j}})\ge \frac{1}{2}a_j^2$.

Let $A_{j+1}=A_j\cap T^{-k^j_{t_2^1}+k^j_{t_1^1} }A_j$ and
$a_{j+1}=\mu(A_{j+1})=\mu(E^j_{k^j_{t_1^j}}\cap
E^j_{k^j_{t_2^j}})\ge \frac{1}{2}a_j^2$. Let
$E^{j+1}_i=T^{-k^{j+1}_i}A_{j+1}$, $i\in \N$, and let
$N_{j+1}=N(a_{j+1},2,\frac{1}{2}a_{j+1}^2)$ be as in Proposition
\ref{infinite-type}. Then for $n\ge N_{j+1}$ and any tuple $\{
s_1<s_2<\cdots <s_n\}$ there exist $1\le t_1^{j+1}<t_2^{j+1}\le n$
with $\mu(E^{j+1}_{s_{t_1^{j+1}}}\cap E^{j+1}_{s_{t_2^{j+1}}})\ge
\frac{1}{2}a_{j+1}^2.$

\medskip

Note that the choices of $\{N_i\}$ is independent of
$\{k_i^j\}_{i=1}^\infty$. \hfill $\square$
\end{rem}

Now we are ready to show

\begin{prop}\label{longproof}
The following statements hold.
\begin{enumerate}
\item For each $d\in\N$, $\F_{fip}\subset\F_{P_d},$ which implies that
$\F_{P_\infty}=\F_{fip}.$

\item $\F_{SG_d}\subset \F_{P_d}$ for each $d\in\N\cup\{\infty\}$.
Moreover one has $\F_{fSG_d}\subset \F_{P_d}$.

\end{enumerate}
\end{prop}
\begin{proof}
(1) Let $F\in \F_{fip}$. Fix $d\in \N$. Now we  show $F\in
\F_{P_d}$. For this purpose, assume that $(X,\mathcal{B},\mu,T)$ is
a measure preserving system and $A\in \mathcal{B}$ with positive
measure.  Since  $F\in \F_{fip}$, there are $p_1,p_2,\cdots p_{\ell_
d}\in \mathbb{Z}$  with $\ell_d=d\sum_{i=1}^d N_i$ such that
$F\supset FS\{p_i\}_{i=1}^{\ell_d}$, where $N_i$ are chosen as in
Remark \ref{remark} for $(X,\mathcal{B},\mu,T)$ and $A$.

Let $A_1=A$. For $p_1,p_1+p_2,\cdots, p_1+\cdots+p_{N_1}$ by the
argument in Remark \ref{remark} there is
$q_1=p_{i_1^1}+\cdots+p_{i_2^1}$ such that $\mu(A_1\cap
T^{-q_1}A_1)\ge\frac{1}{2}a_1^2,$ where $a_1=\mu(A_1)$ and $1\le
i_1^1<i_2^1\le N_1.$ Let $A_2=A_1\cap T^{-q_1}A_1$ and
$a_2=\mu(A_2)$. For $p_{N_1+1},p_{N_1+1}+p_{N_1+2}, \cdots,
p_{N_1+1}+\cdots+p_{N_1+N_2}$, there is
$q_2=p_{i_1^2}+\cdots+p_{i_2^2}$ such that $\mu(A_2\cap
T^{-q_2}A_2)\ge\frac{1}{2}a_2^2,$ where $N_1+1\le i_1^2<i_2^2\le
N_1+N_2.$ Note that $q_1,q_2, q_1+q_2\in F$.

Inductively we obtain $$N_1+\ldots+N_j+1\le i_1^{j+1}<i_2^{j+1}\le
N_1+\ldots+N_{j+1},\ 0\le j\le d-1.$$ $q_1,\ldots, q_d$ and $A_1,
\ldots, A_q$ with $q_j=\sum_{i=i_1^j}^{i_2^j}p_i$ and
$A_j=A_{j-1}\cap T^{-q_{j-1}}A_{j-1}$, $a_j=\mu(A_j)$ such that
$\mu(A_j\cap T^{-q_j}A_j)\ge \frac{1}{2}a_j^2$. Thus
$$\mu(A\cap \bigcap_{n\in FS(\{q_i\}_{i=1}^d)}T^{-n}A)\ge\frac 12 a_d^2>0,$$
and it is clear that $F\supset FS(\{q_i\}_{i=1}^d)$. This implies
that $F\in \F_{P_d}$.

Thus  $\F_{P_\infty}\supset\F_{fip}.$ Since it is clear that
$\F_{P_\infty}\subset\F_{fip},$ we are done.


\medskip
(2) Since each $SG_1$-set is a $\Delta$-set, and hence it is a
Poincar\'e sequence (this is easy to be checked by Poincar\'e
recurrence Theorem \cite{F81}). We First show the case when $d=2$
which will illustrate the general idea. Then we give the proof for
the general case.

Let $F\in SG_2$. Then there is $P=\{P_i\}_{i=1}^\infty\subset \Z$
with $F=SG_2(P)$. Let $(X, \mathcal{B},\mu,T)$ be a m.d.s. and
$A\in\mathcal{B}$ with $\mu(A)>0$. Set $A_1=A$ and $a_1=\mu(A_1)$.

Let $$q_1=\sum_{i=1}^{N_2}P_{2i-1},
q_2=\sum_{i=N_2+1}^{2N_2}P_{2i-1},\ \ \ldots,\ \text{and}\
q_{N_1}=\sum_{i=(N_1-1)N_2+1}^{N_1N_2}P_{2i-1},$$ where
$N_1=N(a_1,2,\frac{1}{2}a_1^2)$ and $N_2=N(a_2,2,\frac{1}{2}a_2^2)$
are chosen as in Remark \ref{remark} for $(X,\mathcal{B},\mu,T)$ and
$A$. Consider the sequence $q_1,q_1+q_2, \ldots, q_1+q_2+\ldots
+q_{N_1}$. Then as in Remark \ref{remark} there are $1\le i_1,j_1\le
N_1$ such that $\mu(A_2)\ge \frac{1}{2}\mu(A)^2$, where $A_2=A_1\cap
T^{-n_1}A_1$ and $n_1=\sum_{i=i_1}^{j_1}q_i $. Note that
$$n_1=P_{2(i_1-1)N_2+1}+P_{2(i_1-1)N_2+3}+ \ldots + P_{2j_1N_2-1}.$$

Now consider the sequence
$$P_{2(i_1-1)N_2}, P_{2(i_1-1)N_2}+P_{2(i_1-1)N_2+2},
\ldots,P_{2(i_1-1)N_2}+P_{2(i_1-1)N_2+2}+\ldots+P_{2j_1N_2}.$$ It
has $N_2$ terms. So as in Remark \ref{remark} there are $1\le
i_2,j_2\le N_2$ such that $\mu(A_2\cap T^{-n_2}A_2)>0$, where
$n_2=\sum_{i=(i_1-1)N_2+i_2}^{(i_1-1)N_2+j_2} P_{2i} $. Note that
$n_1,n_2,n_1+n_2\in F$ by the definition of $SG_2(P)$. It is easy to
verify that
$$\mu(A\cap T^{-n_1}A\cap T^{-n_2}A\cap T^{-n_1-n_2}A)\ge \frac 12
\mu(A_2)^2>0.$$ Hence $F\in \F_{P_2}$.

\medskip

Now we show the general case. Assume that $d\ge 3$ and let $F\in
SG_d$. We show that $F\in \F_{P_d}$.

Since $F\in SG_d$, there is $P=\{P_i\}_{i=1}^\infty\subset \Z$ with
$F=SG_d(P)$. Let $(X, \mathcal{B},\mu,T)$ be a measure preserving
system and $A\in\mathcal{B}$ with $\mu(A)>0$. Set $A_1=A$. Let
$N_1,\ldots,N_d$ be the numbers as defined in Remark \ref{remark}
for $(X,\mathcal{B},\mu,T)$, $A$ and let $M_i=\prod_{j=i}^d N_j$ for
$1\le i\le d$.

Let $$q_1^1=\sum_{i=1}^{M_2}P_{di-(d-1)},\
q_2^1=\sum_{i=M_2+1}^{2M_2}P_{di-(d-1)},\ \ldots,\
q_{N_1}^1=\sum_{i=(N_1-1)M_2+1}^{M_1}P_{di-(d-1)}.$$ Consider the
sequence $q_1^1,q_1^1+q_2^1, \ldots, q_1^1+q_2^1+\ldots +q_{N_1}^1$.
Then as in Remark \ref{remark} there are $1\le i_1,j_1\le N_1$ such
that $\mu(A_2)\ge \frac{1}{2}\mu(A_1)^2$, where $A_2=A_1\cap
T^{-n_1}A_1$ and $n_1=\sum_{i=i_1}^{j_1}q_i^1$.

Let $m_1=(i_1-1)M_2$. Note that there is $t_1\ge M_2-1$ such that
$$n_1=\sum_{i=i_1}^{j_1}q_i^1=P_{dm_1+1}+P_{dm_1+d+1} +\ldots +
P_{dm_1+t_1d+1}.$$

Now consider $$q_1^2=\sum_{i=m_1+1}^{m_1+M_3}P_{di-(d-2)},\
q_2^2=\sum_{i=m_1+M_3+1}^{m_1+2M_3}P_{di-(d-2)},\ \ldots,\
q_{N_2}^2=\sum_{i=m_1+(N_2-1)M_3+1}^{m_1+M_2}P_{di-(d-2)}.$$

Now consider $q_1^2,q_1^2+q_2^2, \ldots, q_1^2+q_2^2+\ldots
+q_{N_2}^2$. It has $N_2$ terms. So as in Remark \ref{remark} there
are $1\le i_2,j_2\le N_2$ such that $\mu(A_3)\ge
\frac{1}{2}\mu(A_2)^2$, where $A_3=A_2\cap T^{-n_2}A_2$ and
$n_2=\sum_{i=i_2}^{j_2} q_i^2$. Let $m_2=m_1+(i_2-1)M_3$. Note that
$n_1,n_2,n_1+n_2\in F$ and there is $t_2\ge M_3-1$ such that
$$n_2=\sum_{i=i_2}^{j_2}q_i^2=P_{dm_2+2}+P_{dm_2+d+2} +\ldots +
P_{dm_2+t_2d+2}.$$ Note that $n_2$ has at least $M_3$ terms.

Inductively for $1\le k\le d-1$ we have $1\le i_k,j_k\le N_k$ and
$$n_k=\sum_{i=i_k}^{j_k}q_i^k=P_{dm_k+k}+P_{dm_k+d+k} +\ldots +
P_{dm_k+t_kd+k},$$ where $t_k\ge M_{k+1}-1$. Also we have
$A_k=A_{k-1}\cap T^{-n_{k-1}}A_{k-1}$ with $\mu(A_k)\ge \frac 12
\mu(A_{k-1})^2$, and $FS(\{n_j\}_{j=1}^{k})\subset F$.

Especially, when $k=d$, we get $1\le i_d<j_d\le N_d$ and
$n_d=\sum_{i=i_d}^{j_d}P_{di}$. By the definition of $SG_d$ we get
that $FS(\{n_i\}_{i=1}^d)\subset F$. From the definition of $A_j,
j=1,2,\ldots, d$, one has
$$\mu(A\cap \bigcap_{n\in FS(\{n_i\}_{i=1}^d)}T^{-n}A)\ge\frac 12 \mu(A_d)^2>0,$$
which implies that $F\in \F_{P_d}$. The proof is completed.
\end{proof}

\subsubsection{}

\noindent{\it {Proof of Proposition \ref{birkhoff}:}} It is clear
that $\F_{B_\infty}\subset\F_{fip}.$ Since $\F_{fip}\subset
\F_{P_\infty}\subset\F_{B_\infty}$ (by Proposition \ref{longproof}
and the obvious fact that $\F_{P_d}\subset \F_{B_d}$) we have
$\F_{B_\infty}=\F_{fip}.$

\subsection{Poincar\'e recurrence sets and $\RP^{[d]}$}

\begin{thm}\label{poincare} Let $(X,T)$ be a minimal t.d.s.. Then
for each $d\in\N\cup \{\infty\}$, $(x,y)\in \RP^{[d]}$ if and only
if $N(x,U)\in \F_{P_d}$ for any neighborhood $U$ of $y$.
\end{thm}

\begin{proof} We first show the case when $d\in\N$.
$(\Leftarrow)$ Since $\F_{P_d}\subset \F_{B_d}$, it follows from
Theorem \ref{birkhoff}. Or one proves it directly as follows. Assume
$N(x,U)\in \F_{P_d}$ for any neighborhood $U$ of $y$ and $\mu\in
M(X,T)$. Then $\text{supp}(\mu)=X$ since $(X,T)$ is minimal. For any
$\epsilon>0$, let $U_1=B(x,\epsilon)$ and
$U_2=B(y,\frac{\epsilon}{2})$. Since $N(x,U_2)$ is a Poincar\'e
sequence of order $d$ and $\mu(U_2)>0$ there exist
$n_1,\ldots,n_d\in \mathbb{Z}$ such that $T^n x\in U_2$ for $n\in
FS(\{n_i\}_{i=1}^d)$ and $\mu(U_2\cap \bigcap_{n\in
FS(\{n_i\}_{i=1}^d)}T^{-n}U_2)>0$. Then any $y'\in U_2\cap
\bigcap_{n\in FS(\{n_i\}_{i=1}^d)}T^{-n}U_2$ satisfies $T^n y'\in
U_2$ for any $n\in FS(\{n_i\}_{i=1}^d)$. Thus, $\rho(y,y')<\epsilon$
and $\rho(T^mx,T^my')\le \text{diam}(U_2)<\epsilon$ for any $n\in
FS(\{n_i\}_{i=1}^d)$, which imply that $(x,y)\in \RP^{[d]}$.

\medskip
$(\Rightarrow)$ Assume that $(x,y)\in \RP^{[d]}$ and $U$ is a
neighborhood of $y$. By Theorem \ref{huang12}, $N(x,U)\in
\F_{SG_d}$. Then by Proposition \ref{longproof} we have
$N(x,U)\in\F_{P_d}.$

\medskip

The case $d=\infty$ follows from the case $d\in \N$ and definitions.
\end{proof}

\subsection{Conclusion}

Now we sum up the results of this section and previous two sections.
Note that $\F_{Bir_\infty}$ and $\F_{Poi_\infty}$ can be defined
naturally. Since $\F_{1,0}\subset \F_{2,0}\subset \ldots$ we define
$\F_{\infty,0}=:\bigcup_{d=1}^\infty \F_{d,0}$. Another way to do
this is that one follows the idea in \cite{D-Y} to define
$\infty$-step nilsystems and view $\F_{\infty,0}$ as the family
generated by all Nil$_\infty$ Bohr$_0$-sets. It is easy to check
that Theorem \ref{several} holds for $d=\infty$.

Thus we have
\begin{thm}\label{rpd}
Let $(X,T)$ be a minimal t.d.s. and $x,y\in X$. Then the following
statements are equivalent for $d\in\N\cup\{\infty\}$:

\begin{enumerate}

\item $(x,y)\in \RP^{[d]}$.

\item $N(x,U)\in \F_{d,0}^*$ for each
neighborhood $U$ of $y$.

\item $N(x,U)\in \F_{Poi_d}$ for each
neighborhood $U$ of $y$.

\item $N(x,U)\in \F_{Bir_d}$ for each neighborhood $U$ of $y$.

\item $N(x,U)\in \F_{SG_d}$ for each neighborhood $U$ of $y$.

\item $N(x,U)\in \F_{fSG_d}$ for each neighborhood $U$ of $y$.

\item $N(x,U)\in \F_{B_d}$ for each
neighborhood $U$ of $y$.

\item $N(x,U)\in \F_{P_d}$ for each neighborhood
$U$ of $y$.

\end{enumerate}

\end{thm}

\section{$d$-step almost automorpy and recurrence sets}
In the previous sections we give some characterizations of
regionally proximal relation of order $d$. In the present section we
introduce and study $d$-step almost automorpy.

\subsection{Definition of $d$-step almost automorpy}

\subsubsection{}
First we recall the notion of $d$-step almost automorphic systems
and give its structure theorem.

\begin{de}
Let $(X,T)$ be a t.d.s. and $x\in X$, $d\in \N\cup\{\infty\}$.  $x$
is called an {\em $d$-step almost automorphic} point (or $d$-step AA
for short) if $\RP^{[d]}(Y)[x]=\{x\}$, where
$Y=\overline{\{T^nx:n\in \mathbb{Z}\}}$ and $\RP^{[d]}(Y)[x]=\{y\in
Y: (x,y)\in \RP^{[d]}(Y)\}$.

A minimal t.d.s. $(X,T)$ is called {\em $d$-step almost automorphic}
if it has a $d$-step almost automorphic point.
\end{de}

\begin{rem} Since
$$\RP^{[\infty]}\subseteq \ldots \subseteq \RP^{[d]}\subseteq \RP^{[d-1]}\subseteq
\ldots \subseteq \RP^{[1]},$$ we have
$$\text{AA}=\text{1-step AA}\Rightarrow \ldots \Rightarrow\text{(d-1)-step AA}
\Rightarrow\text{d-step AA}\Rightarrow \ldots \Rightarrow
\infty\!-\!\text{step AA}.$$
\end{rem}

\subsubsection{}
The following theorem follows from Theorem \ref{ShaoYe}.

\begin{thm}\label{thm-AA}
Let $(X,T)$ be a minimal t.d.s.. Then $(X,T)$ is a $d$-step almost
automorphic system for some $d\in \N\cup\{\infty\}$ if and only if
it is an almost one-to-one extension of its maximal $d$-step
nilfactor $(X_d, T)$.

\[
\begin{CD}
X @>{T}>> X\\
@V{\pi}VV      @VV{\pi}V\\
X_d @>{T }>> X_d
\end{CD}
\]

\end{thm}

\medskip

\subsection{$1$-step almost automorphy}
First we recall some classical results about almost automorphy.

Let $(X,T)$ be a minimal t.d.s.. In \cite{V68} it is proved that
$(x,y)\in \RP^{[1]}$ if and only if for each neighborhood $U$ of
$y$, $N(x,U)$ contains some $\D$-set, see also Theorem
\ref{huang12}. Similarly, we have for a minimal system $(X,T)$,
$(x,y)\in \RP^{[1]}$ if and only if for each neighborhood $U$ of
$y$, $N(x,U)\in \F_{Poi}$ \cite{HLY}, see also Theorem \ref{rpd}.

Using these theorems and the facts that $\F_{Poi}$ and $\F_{Bir}$
have the Ramsey property, one has

\begin{thm}
Let $(X,T)$ be a minimal t.d.s. and $x\in X$. Then the following
statements are equivalent:
\begin{enumerate}
\item $x$ is AA.

\item $N(x,V)\in \F_{Poi}^*$ for each
neighborhood $V$ of $x$.

\item $N(x,V)\in \F_{Bir}^*$ for each
neighborhood $V$ of $x$.

\item \cite{F} $N(x,V)\in \Delta^*$ for each
neighborhood $V$ of $x$.
\end{enumerate}
\end{thm}

We will not give the proof of this theorem since it is the special
case of Theorem~ \ref{AAgeneral}.

\subsection{$\infty$-step almost automorphy}

In this subsection we give one characterization for $\infty$-step
AA. Followed from Theorem \ref{ShaoYe}, one has

\begin{thm}\label{shaoye}
Let $(X,T)$ be a minimal t.d.s. and $d\ge 1$. Then

\begin{enumerate}
\item $(x,y)\in \RP^{[d]}$ if and only if $N(x,U)$ contains a finite IP-set of
length $d+1$ for any neighborhood $U$ of $y$, and thus

\item $(x,y)\in \RP^{[\infty]}$ if and only if $N(x,U)\in \F_{fip}$ for any
neighborhood $U$ of $y$.
\end{enumerate}
\end{thm}

To show the next theorem we need the following lemma which should be
known, see for example Huang, Li and Ye \cite{HLY2}.

\begin{lem}
$\F_{fip}$ has the Ramsey property.
\end{lem}

We have the following

\begin{thm}\label{aainfty}
Let $(X,T)$ be a minimal t.d.s.. Then $(X,T)$ is $\infty$-step AA if
and only if there is $x\in X$ such that $N(x,V)\in \F_{fip}^*$ for
each neighborhood $V$ of $x$.
\end{thm}

\begin{proof}
Assume that there is $x\in X$ such that $N(x,V)\in \F_{fip}^*$ for
each neighborhood $V$ of $x$. If there is $y\in X$ such that
$(x,y)\in\RP^{[\infty]}$, then by Proposition \ref{shaoye} for any
neighborhood $U$ of $y$, $N(x,U)\in\F_{fip}$. This implies that
$x=y$, i.e. $(X,T)$ is $\infty$-step AA.

Now assume that $(X,T)$ is $\infty$-step AA, i.e. there is $x\in X$
such that $\RP^{[\infty]}[x]=\{x\}$. If for some neighborhood $V$ of
$x$, $N(x,V)\not\in \F_{fip}^*$, then $N(x,V^c)$ contains finite
IP-sets of arbitrarily long lengths.

Let $U_1=V^c$. Covering $U_1$ by finitely many closed balls
$U_1^1,\ldots, U_1^{i_1}$ of diam $\le 1$. Then there is $j_1$ such
that $N(x,U_1^{j_1})$ contains finite IP-sets of arbitrarily long
lengths. Let $U_2=U_1^{j_1}$. Covering $U_1$ by finitely many closed
balls $U_2^1,\ldots, U_2^{i_2}$ of diam $\le \frac{1}{2}$. Then
there is $j_2$ such that $N(x,U_2^{j_2})$ contains finite IP-sets of
arbitrarily long lengths. Let $U_3=U_2^{j_2}$. Inductively, there
are a sequence of closed balls $U_n$ with diam $\le \frac{1}{n}$
such that $N(x,U_n)$ contains finite IP-sets of arbitrarily long
lengths. Let $\{y\}=\bigcap U_n$. It is clear that $(x,y)\in
\RP^{[\infty]}$ with $y\not=x$, a contradiction. Thus $N(x,V)\in
\F_{fip}^*$ for each neighborhood $V$ of $x$.
\end{proof}

\subsection{Characterization of $d$-step almost automorphy}

Now we use the results built in previous sections to get the
following characterization for $d$-step AA via recurrence sets.

\begin{thm}\label{AAgeneral}
Let $(X,T)$ be a minimal t.d.s., $x\in X$ and
$d\in\N\cup\{\infty\}$. Then the following statements are
equivalent:
\begin{enumerate}
\item $x$ is $d$-step AA point.

\item $N(x,V)\in \F_{d,0}$ for each neighborhood $V$ of $x$.

\item $N(x,V)\in \F_{Poi_d}^*$ for each
neighborhood $V$ of $x$.

\item $N(x,V)\in \F_{Bir_d}^*$ for each
neighborhood $V$ of $x$.

\end{enumerate}
\end{thm}

\begin{proof}
Roughly speaking the theorem follows from Theorem \ref{rpd}, the
fact $\F_{d,0}^*, \F_{Poi_d}$ and $\F_{Bir_d}$ have the Ramsey
property, and the idea of the proof of Theorem \ref{aainfty}. We
show that $(1)\Leftrightarrow(2)$, and the rest is similar.

$(1)\Rightarrow (2)$: Let $x$ be a $d$-step AA point. If (2) does
not hold, then there is some neighborhood $V$ of $x$ such that
$N(x,V)\not \in \F_{d,0}$. Then $N(x,V^c)=\Z\setminus N(x,V)\in
\F^*_{d,0}$. Since $\F^*_{d,0}$ has the Ramsey property, similar to
the proof of Theorem \ref{aainfty} one can find some $y\in V^c$ such
that $N(x,U)\in \F^{*}_{d,0}$ for every neighborhood $U$ of $y$. By
Theorem \ref{rpd}, $y\in \RP^{[d]}[x]$. Since $y\neq x$, this
contradicts to the fact $x$ being $d$-step AA.

$(2)\Rightarrow (1)$: If $x$ is not $d$-step AA, then there is some
$y\in \RP^{[d]}[x]$ with $x\neq y$. Let $U_x$ and $U_y$ be
neighborhoods of $x$ and $y$ with $U_x\cap U_y=\emptyset$. By $(2)$
$N(x,U_x)\in \F_{d,0}$. By Theorem \ref{rpd}, $N(x,U_y)\in
\F^*_{d,0}$. Hence $N(x,U_x)\cap N(x,U_y)\neq \emptyset$,  which
contradicts the fact that $U_x\cap U_y=\emptyset$.
\end{proof}

\subsection{Some further questions}
(1) We have defined and studied $d$-recurrence and Poincar\'e
sequence of order $d$; and $d$-topological recurrence and Birkhoff
recurrence set of order $d$. It is not clear the relation between
$\F_{P_d}$ and $\F_{Poi_d}$; and $\F_{B_d}$ and $\F_{Bir_d}$. Also
it will be very interesting if one can show that $\F_{B_d}\subset
\F_{d,0}^*$ which implies that $x$ is $d$-step AA if and only if
$N(x,V)\in \F_{B_d}^*$ for each neighborhood $V$ of $x$ by Theorem~
\ref{AAgeneral}.

\medskip

(2) In \cite{V65}  Veech proved that for a minimal t.d.s. $(X, T)$,
a point $x\in X$ is almost automorphic if and only if from any
sequence $\{n_i'\}\subseteq \Z$ one may extract a subsequence
$\{n_i\}$ such that $\lim_{i\to \infty} T^{n_i}x=y$ for some $y\in
X$ and $\lim_{i\to \infty} T^{-n_i}y=x.$ We do not know if there is
a similar characterization for $d$-step almost automorphic points
for $d\ge 2$.

\appendix

\section{The Ramsey properties} \label{appendix:Ramsey}

Recall that a family $\F$ has the Ramsey property means that if
$A\in\F$ and $A=\cup_{i=1}^n A_i$ then one of $A_i$ is still in
$\F$. In this section, we show that $\F_{SG_2}$ does not have the
Ramsey property.

\begin{thm}
$\F_{SG_2}$ does not have the Ramsey property.
\end{thm}

\begin{proof}
Let $P=\{p_1, p_2, \ldots \}$ be a subsequence of $\N$ with
$p_{i+1}>2(p_1+\ldots +p_i)$. The assumption that
$p_{i+1}>2(p_1+\ldots +p_i)$ ensures that each element of $SG_2(P)$
has a unique expression with the form of $\sum_i p_{j_i}$.

Now divide the set $SG_2(P)$ into the following three sets:
\begin{equation*}
\begin{split}
B_1&=\{p_{2n-1}+\ldots+p_{2m-1}:n\le m\in\N\}=SG_1(\{p_1,p_3,\ldots\}),\\
B_2&=\{p_{2n}+\ldots+p_{2m}:n\le m\in\N\}=SG_1(\{p_2,p_4,\ldots\}),\\
B_0&=SG_2(P)\setminus (B_1\cup B_2).
\end{split}
\end{equation*}

We show that $B_i\not\in \F_{SG_2}$ for $i=0,1,2.$  In fact, we will
prove that for each $i=0,1,2$ there do not exist $a_1<a_2<a_3$ such
that
\begin{equation}
    a_1,a_2,a_3,a_1+a_2,a_2+a_3,a_1+a_3\subseteq B_i, \tag{$*$}
\end{equation}
which obviously implies that $B_i\not\in \F_{SG_2}$ for $i=0,1,2.$
\medskip

\noindent (1). First we show $B_2\not\in \F_{SG_2}$. The proof
$B_1\not\in \F_{SG_2}$ follows similarly. Assume the contrary, i.e.
there exist $a_1<a_2<a_3$ such that
\begin{equation*}
    a_1,a_2,a_3,a_1+a_2,a_2+a_3,a_1+a_3\subseteq B_2.
\end{equation*}
Let
\begin{equation*}
\begin{split}
a_1&=p_{2n_1}+\ldots+p_{2m_1},\ n_1\le m_1;\\
a_2&=p_{2n_2}+\ldots+p_{2m_2},\ n_2\le m_2;\\
a_3&=p_{2n_3}+\ldots+p_{2m_3},\ n_3\le m_3.
\end{split}
\end{equation*}
Since $a_1<a_2<a_3$ and the assumption that $p_{i+1}>2(p_1+\ldots
+p_i)$, one has that $m_1\le m_2\le m_3$. Since $a_1+a_2, a_2+a_3\in
B_2$, one has that $n_2=m_1+1$ and $n_3=m_2+1$. Hence $n_3=m_2+1\ge
n_2+1=m_1+2$, i.e. $n_3>m_1+1$. Thus $$a_1+a_3\not \in B_2,$$ a
contraction!

\medskip

\noindent (2). Now we show $B_0\not\in \F_{SG_2}$. Assume the
contrary, i.e. there exist $a_1<a_2<a_3$ such that
\begin{equation*}
    a_1,a_2,a_3,a_1+a_2,a_2+a_3,a_1+a_3\subseteq B_0.
\end{equation*}
Let
\begin{equation*}
\begin{split}
a_1&=p_{i^1_1}+p_{i^1_2}+\ldots+p_{i^1_{k_1}};\\
a_2&=p_{i^2_1}+p_{i^2_2}+\ldots+p_{i^2_{k_2}};\\
a_3&=p_{i^3_1}+p_{i^3_2}+\ldots+p_{i^3_{k_3}},
\end{split}
\end{equation*}
where $i^r_1<i^r_2<\ldots <i^r_{k_r}$, $i^r_{j+1}\le i^r_j+2$ for
$1\le j\le k_r-1$, and there are both even and odd numbers in
$\{i^r_1,i^r_2,\ldots,i^r_{k_r}\}$ $(r=1,2,3)$.

Since there are both even and odd numbers in
$\{i^r_1,i^r_2,\ldots,i^r_{k_r}\}$ $(r=1,2,3)$ and $i^r_{j+1}\le
i^r_j+2$ for $1\le j \le k_r-1$, there exist $1\le j_r\le k_r-1$
such that $i^r_{j+1}=i^r_{j_r}+1$. Since $a_1<a_2<a_3$ and the
assumption that $p_{i+1}>2(p_1+\ldots +p_i)$, one has that
$i^1_{k_1}\le i^2_{k_2}\le i^3_{k_3}$. Note that we have
\begin{equation*}
    i^1_1<i^1_2<\ldots <i^1_{j_1}<i^1_{j_1+1}=i^1_{j_1}+1<\ldots<i^1_{k_1},
\end{equation*}
\begin{equation*}
    i^2_1<i^2_2<\ldots <i^2_{j_2}<i^2_{j_2+1}=i^2_{j_2}+1<\ldots<i^2_{k_2},
\end{equation*}
\begin{equation*}
    i^3_1<i^3_2<\ldots <i^3_{j_3}<i^3_{j_3+1}=i^3_{j_3}+1<\ldots<i^3_{k_3}.
\end{equation*}

The condition $a_1+a_2\in B_0$ implies that
\begin{equation}
    i^1_{j_1+1}<i^2_1\le i^1_{k_1}+2; \ i^1_{k_1}<i^2_{j_2}. \tag{a}
\end{equation}
In fact if $i^2_1<i^1_{j_1}$, then the gap $\{i^1_{j_1},
i^1_{j_1}+1\}$ is missing in the term of $a_2$ and it contradicts
the assumption $a_2\in \F_{SG_2}$. The statement
$i^1_{k_1}<i^2_{j_2}$ follows by the same argument.

Similarly, using the assumptions $a_2+a_3\in B_0$ and $a_1+a_3\in
B_0$, one has
\begin{equation}
    i^2_{j_2+1}<i^3_1\le i^2_{k_2}+2; \ i^2_{k_2}<i^3_{j_3}. \tag{b}
\end{equation}
and
\begin{equation}
    i^1_{j_1+1}<i^3_1\le i^1_{k_1}+2; \ i^1_{k_1}<i^3_{j_3}. \tag{c}
\end{equation}
From (a), we have that $i_{k_1}^1<i^2_{j_2}$; and from (b), we have
$i^2_{j_2+1}=i^2_{j_2}+1< i^3_1$. Hence we have $i^3_1\ge
i^1_{k_1}+3$, which contradicts (c). The proof is completed.
\end{proof}

\section{Compact Hausdorff Systems}\label{appendix:CT2}

In this section we discuss compact Hausforff systems, i.e. the
systems with phase space being compact Hausdorff. The reason for
this is not generalization for generalization's sake, but rather
that we have to deal with non-metrizable systems. For example, we
will use (in the proof of Theorem \ref{huang10}) an important tool
named Ellis semigroup which is a subspace of an uncountable product
of copies of the phase space and therefore in general not
metrizable.

\subsection{Compact Hausdorff systems}

In the classical theory of abstract topological dynamics, the basic
assumption about the system is that the space is a compact Hausdorff
space and the action group is a topological group. In this paper, we
mainly consider the compact metrizable system under $\Z$-actions,
but in some occasions we have to deal with compact Hausdorff spaces
which are non-metrizable. Note that each compact Hausdorff space is
a uniform space, and one may use the uniform structure replacing the
role of a metric, see for example the Appendix of \cite{Au88}.

First we recall a classical equality concerning regionally proximal
relation in compact Hausdorff systems.
A {\em compact Hausdorff system} is a pair $(X,T)$, where $X$ is a
compact Hausdorff space and $T:X\rightarrow X$ is a homeomorphism.
Let $(X,T)$ be a compact Hausdorff system and $\mathcal{U}_X$ be the
unique uniform structure of $X$. 
The {\em regionally proximal relation} on $X$ is defined
by
\begin{equation*}
    \RP=\bigcap_{\a \in \mathcal{U}_X}\overline{ \bigcup_{n\in \Z}(T\times
    T)^{-n}\a}
\end{equation*}

\subsection{Ellis semigroup}

A beautiful characterization of distality was given by R. Ellis
using so-called enveloping semigroup. Given a compact Hausdorff
system $(X,T)$, its {\it enveloping semigroup} (or {\em Ellis
semigroup}) $E(X,T)$ is defined as the closure of the set $\{T^n: n
\in \Z\}$ in $X^X$ (with its compact, usually non-metrizable,
pointwise convergence topology). Ellis showed that a compact
Hausdorff system $(X,T)$ is distal if and only if $E(X,T)$ is a
group if and only if every point in $(X^2, T\times T)$ is minimal
\cite{E69}.

\subsection{Limits of Inverse systems}

Suppose that  every $\ll$ in a set $\Lambda$ directed by the
relation $\le$ corresponds a t.d.s. $(X_\ll,T_\ll)$, and that for
any $\ll,\xi\in \Lambda$ satisfying $\xi\le \ll$ a factor map
$\pi^\ll_\xi: (X_\ll,T_\ll)\rightarrow (X_\xi,T_\xi)$ is defined;
suppose further that $\pi^\xi_\tau\pi^\ll_\xi=\pi^\ll_\tau$ for all
$\ll, \xi,\tau\in \Lambda$ with $\tau\le \xi\le \ll$ and that
$\pi^\ll_\ll={\rm id}_X$ for all $\ll\in \Lambda$. In this situation
we say that the family $\{X_\ll,
\pi^\ll_\xi,\Lambda\}=\{(X_\ll,T_\ll), \pi^\ll_\xi,\Lambda\}$ is an
{\em inverse system of the systems $(X_\ll,T_\ll)$}; and the
mappings $\pi^\ll_\xi$ are called {\em bonding mappings} of the
inverse system.

Let $\{X_\ll, \pi^\ll_\xi,\Lambda\}$ be an inverse system. The {\em
limit of the inverse system  $\{X_\ll, \pi^\ll_\xi,\Lambda\}$} is
the set
\begin{equation*}
    \Big\{(x_\ll)_\ll\in \prod_{\ll\in \Lambda}X_\ll: \pi^\ll_\xi(x_\ll)=
    x_\xi\ \text{for all $\xi\le \ll\in \Lambda$}\Big\},
\end{equation*}
and is denoted by $\varprojlim \{X_\ll, \pi^\ll_\xi,\Lambda\}$. Let
$X=\varprojlim \{X_\ll, \pi^\ll_\xi,\Lambda\}$. For each $\ll\in
\Lambda $, let $\pi_\ll: X\rightarrow X_\ll,
(x_\sigma)_\sigma\mapsto x_\ll$ be the projection mapping.

\medskip

A well known result is the following (see for example \cite{Ke}):

\begin{lem}\label{inverse}
Each compact Hausdorff system is the inverse limit of topological
dynamical systems.
\end{lem}

\subsection{The regionally proximal relation of order $d$ for
compact Hausdorff systems}

The definition of the regionally proximal relation of order $d$ for
compact Hausdorff systems is similar to the metric case.

\begin{de}
Let $(X, T)$ be a compact Hausdorff system, $\mathcal{U}_X$ be the
unique uniform structure of $X$ and let $d\ge 1$ be an integer. A
pair $(x, y) \in X\times X$ is said to be {\em regionally proximal
of order $d$} if for any $\a  \in \mathcal{U}_X$, there exist $x',
y'\in X$ and a vector ${\bf n} = (n_1,\ldots , n_d)\in\Z^d$ such
that $(x, x') \in  \a, (y, y') \in \a$, and $$ (T^{{\bf n}\cdot
\ep}x', T^{{\bf n}\cdot \ep}y') \in \a \ \text{for any $\ep\in
\{0,1\}^d$, $\ep\not=(0,\ldots,0)$},
$$ where ${\bf n}\cdot \ep = \sum_{i=1}^d \ep_in_i$. The set of all
regionally proximal pairs of order $d$ is denoted by $\RP^{[d]}(X)$,
which is called {\em the regionally proximal relation of order $d$}.
\end{de}

By Lemma \ref{inverse}, each compact Hausdorff system is the inverse
limit of topological dynamical systems. Recall the definition of the
product uniformity. Let $(X_\ll, \U_\ll)_{\ll\in \Lambda}$ be a
family of uniform spaces and let $Z=\prod_{\ll\in \Lambda}X_\ll$.
The uniformity on $Z$ (the product uniformity) is defined as
follows. If $F=\{\ll_1,\ldots,\ll_m\}$ is a finite subset of the
index set $\Lambda$ and $\a_{\ll_j}\in \U_{\ll_j}$ $(j=1,\ldots,m)$,
let
$$\Phi_{\a_{\ll_1},\ldots,\a_{\ll_m}}=\{(x,y)\in Z\times Z:
(x_{\ll_j},y_{\ll_j})\in \a_{\ll_j} ,\ j=1,\ldots, m\}.$$ The
collection of all such sets $\Phi_{\a_{\ll_1},\ldots,\a_{\ll_m}}$
for all finite subsets $F$ of $\Lambda$ is a base for the product
uniformity. From this and the definition of the regionally proximal
relation of order $d$, one has the following result.

\begin{prop}
Let $(X,T)$ be a compact Hausdorff system and $d\in \N$. Suppose
that $X=\varprojlim \{X_\ll, \pi^\ll_\xi,\Lambda\}$, where $(X_\ll,
T_\ll)_{\ll\in \Lambda}$ are t.d.s.. Then
$$\RP^{[d]}(X)=\varprojlim
\{\RP^{[d]}(X_\ll), \pi^\ll_\xi\times \pi^\ll_\xi,\Lambda\}.$$
\end{prop}

Thus combining this proposition with Theorem \ref{ShaoYe}, one has

\begin{thm}
Let $(X, T)$ be a minimal compact Hausdorff system and $d\in \N$.
Then
\begin{enumerate}

\item $\RP^{[d]}(X)$ is an equivalence relation, and so is
$\RP^{[\infty]}.$

\item If $\pi:(X,T)\lra (Y,S)$ is a factor map, then $(\pi\times
\pi)(\RP^{[d]}(X))=\RP^{[d]}(Y).$

\item $(X/\RP^{[d]},T)$ is the maximal nilfactor of $(X,T)$.
\end{enumerate}
\end{thm}

Note that for a compact Hausdorff system $(X,T)$ we say that it is a
system of order $d$ for some $d\in\N$ if it is an inverse limit of
basic $d$-step nilsystems.

\section{Intersective}\label{appendix:Intersective}
It is well known that $P$ is a Birkhoff recurrence set iff
$P\cap(F-F)\not=\emptyset$ for each $F\in\F_{s}$. To give a similar
characterization we have
\begin{de}
A subset $P$ is intersective (topologically) of order $d$ if for
each $F\in \F_{s}$ there are $n_1,\ldots,n_d$ with
$FS(\{n_i\}_{i=1}^d)\subset P$ and $a\in F$ with
$a+FS(\{n_i\}_{i=1}^d)\subset F$, i.e. $F\cap\bigcap_{n\in
FS(\{n_i\}_{i=1}^d)} (F-n)\not=\emptyset$.
\end{de}

\begin{thm}A subset $P$ is intersective (topologically) of order $d$ if and
only if it is a Birkhoff recurrence set of order $d$.
\end{thm}
\begin{proof} Assume that $P$ is a Birkhoff recurrence set of order $d$.
Let $F\in\F_s$. Then $1_F\in \{0,1\}^{\Z_+}$. Let $(X,T)$ be a
minimal subsystem of $(\overline{orb(1_F,T)},T)$, where $T$ is the
shift. Since $F\in \F_s$, $[1]$ is a non-empty open subset of $X$.
By the definition there are $n_1,\ldots,n_d$ with
$FS(\{n_i\}_{i=1}^d)\subset P$ such that $[1] \cap
\big(\bigcap_{n\in FS(\{n_i\}_{i=1}^d)}T^{-n}[1] \big)\neq
\emptyset.$ It implies that there is $a\in F$ with
$a+FS(\{n_i\}_{i=1}^d)\subset F$ and hence $P$ is intersective
(topologically) of order $d$.

\medskip
Assume that $P$ is intersective (topologically) of order $d$. Let
$(X,T)$ be a minimal t.d.s. and $U$ be an open non-empty subsets.
Take $x\in U$, then $F=N(x,U)\in \F_s$. Thus there are
$n_1,\ldots,n_d$ with $FS(\{n_i\}_{i=1}^d)\subset P$ and $a\in F$
with $a+FS(\{n_i\}_{i=1}^d)\subset F$. It follows that $U \cap
\big(\bigcap_{n\in FS(\{n_i\}_{i=1}^d)}T^{-n}U \big)\neq \emptyset.$
\end{proof}

It is well known that $P$ is a Poincar\'e sequence if and only if
$P\cap(F-F)\not=\emptyset$ for each $F\in\F_{pubd}$. To give a
similar characterization we have

\begin{de} A subset $P$ is intersective of order $d$ if for each $F\in
\F_{pubd}$ there are $n_1,\ldots,n_d$ with
$FS(\{n_i\}_{i=1}^d)\subset P$ and $a\in F$ with
$a+FS(\{n_i\}_{i=1}^d)\subset F$.
\end{de}

\begin{thm}
A subset is intersective of order $d$ if and only if it is a
Poincar\'e sequence of order $d$.
\end{thm}
\begin{proof}
Assume that $P$ is intersective of order $d$. Let $(X,\mathcal{B},
\mu,T)$ be a measure preserving system and $A\in\mathcal{B}$ with
$\mu(A)>0$. By the Furstenberg corresponding principle, there exists
$F\subset \mathbb{Z}$ such that $d(F)\ge \mu(A)$ and
\begin{equation}\label{F1}
\{\alpha\in \F(\mathbb{Z}):\bigcap_{n\in \alpha}
(F-n)\not=\emptyset\}\subseteq \{\alpha\in
\F(\mathbb{Z}):\mu(\bigcap_{n\in \alpha} T^{-n}A)>0\},
\end{equation}
where $\F(\mathbb{Z})$ denote the collection of finite non-empty
subsets of $\mathbb{Z}$. Since $P$ is intersective of order $d$,
there are $n_1,\ldots,n_d$ with $FS(\{n_i\}_{i=1}^d)\subset P$ and
$a\in F$ with $a+FS(\{n_i\}_{i+1}^d)\subset F$, i.e. $F\cap \bigcap
\limits_{n\in FS(\{n_i\})} (F-n)\not=\emptyset$. By (\ref{F1}) $P\in
\F_{P_d}$.

Now assume that $P\in\F_{P_d}$ and $F\in \F_{pubd}$. Then by the
Furstenberg corresponding principle, there are a measure preserving
system $(X,\mathcal{B}, \mu,T)$ and $A\in\mathcal{B}$ such that
$\mu(A)=BD^*(F)>0$ and
\begin{equation}\label{F2}
BD^*(\bigcap_{n\in \alpha} (F-n))\ge \mu(\bigcap_{n\in \alpha}
T^{-n}A)
\end{equation}
for all $\alpha\in \F(\Z)$. Since $P\in\F_{P_d}$, there are
$n_1,\ldots,n_d$ with $FS(\{n_i\})\subset P$ and $\mu(A\cap \bigcap
\limits_{n\in FS(\{q_i\}_{i=1}^d)}T^{-n}A)>0.$ This implies $F\cap
\bigcap \limits_{n\in FS(\{n_i\})} (F-n)\not=\emptyset$ by
(\ref{F2}).
\end{proof}



\begin{thebibliography}{SSS}

\bibitem{Ak} E. Akin, \textit{Recurrence in topological dynamical
systems: Furstenberg families and Ellis actions}, Plenum Press, New
York, 1997.

\bibitem{Au88} J. Auslander, \textit{Minimal flows and their
extensions}, North-Holland Mathematics Studies {\bf 153} (1988),
North-Holland, Amsterdam.


\bibitem{Bergelson06} V. Bergelson, \textit{Combinatorial and Diophantine applications of
ergodic theory}, Appendix A by A. Leibman and Appendix B by Anthony
Quas and M\'at\'e Wierdl. Handbook of dynamical systems. Vol. 1B,
745--869, Elsevier B. V., Amsterdam, 2006.

\bibitem{BFW} V. Bergelson, H. Furstenberg and B. Weiss, \textit{Piecewise-Bohr Sets of
Integers and Combinatorial Number Theory}, (Algorithms and
Combinatorics, 26). Springer, Berlin, 2006, pp. 13-37.

\bibitem{BHK05} V. Bergelson, B. Host and  B. Kra, \textit{Multiple recurrence
and nilsequences. With an appendix by Imre Ruzsa}, Invent. Math.,
{\bf 160} (2005), no. 2, 261-303.


\bibitem{BM20} V. Bergelson and R. McCutcheon, \textit{An ergodic IP polynomial
Szemer¨¦di theorem. Mem. Amer. Math. Soc.}, {\bf 146}(2000), no.
695.

\bibitem{Bohr} H. Bohr, \textit{Fastperiodische Funktionen},
Springer-Verlag, Berlin, 1932. (English translation: Almost periodic
functions, Chelsea, 1951.)

\bibitem{Bochner55} S. Bochner, \textit{Curvature and Betti numbers in
real and complex vector bundles}, Universit\'a e Plolitecnico di
Toorino, Rendiconti del  seminario matematico, vol. {\bf 15}
(1955-56), 225-254.

\bibitem{Bochner62} S. Bochner, \textit{A new approach to almost
periodicity}, Proc. Nat. Acad. Sci. USA, {\bf 48} (1962), 2039-2043.

\bibitem{D-Y} P. Dong, S. Donoso, A. Maass, S. Shao and X. Ye, {\it Infinite-step nilsystems,
independence and complexity}, arXiv:1105.3584, Ergod. Th. and Dynam.
Sys., to appear.




\bibitem{E69} R. Ellis, \textit{Lectures on topological dynamics},
W. A. Benjamin, Inc., New York, 1969.

\bibitem{EG} R. Ellis and W. Gottschalk, \textit{ Homomorphisms of transformation
groups}, Trans. Amer. Math. Soc., {\bf 94} (1960), 258-271.




\bibitem{FLW} N. Frantzikinakis, E. Lesigne and M. Wierdl, \textit{Sets of
$k$-recurrence but not $(k+1)$-recurrence}, Ann. Inst. Fourier
(Grenoble), {\bf 56} (2006), no. 4, 839-849.


\bibitem{F77} H. Furstenberg, \textit{Ergodic behavior of diagonal measures and a
theorem of Szemer\'edi on arithmetic progressions}. J. Analyse
Math., {\bf 31} (1977), 204-256.

\bibitem{F81} H. Furstenberg, \textit{Poincar$\acute{e}$ recurrence and number theory},
Bull. Amer. Math. Soc. (N.S.), {\bf 5} (1981), no. 3, 211-234.


\bibitem{F} H. Furstenberg, \textit{Recurrence in ergodic theory and
combinatorial number theory}, M. B. Porter Lectures. Princeton
University Press, Princeton, N.J., 1981.

\bibitem{FK} H. Furstenberg and Y. Katznelson, \textit{An ergodic Szemer¨¦di theorem for
 IP-systems and combinatorial theory}, J. Analyse Math., {\bf 45} (1985), 117-168.


\bibitem{G} J. Gillis, {\it Notes on a property of measurable sets}, J.
Lon. Math. Soc., {\bf 11}(1936), 139-141.

\bibitem{G76} S. Glasner, \textit{Proximal flows}, Lecture Notes in
Mathematics, Vol. {\bf 517}, Springer-Verlag, Berlin-New York, 1976.

\bibitem{G93} E. Glasner, \textit{Minimal nil-transformations of class two}, Israel
J. Math., {\bf 81}(1993), 31-51.

\bibitem{G94} E. Glasner, \textit{Topological ergodic decompositions and
applications to products of powers of a minimal transformation}, J.
Anal. Math., {\bf 64} (1994), 241-262.

\bibitem{GT08} B. Green and T. Tao, \textit{Quadratic uniformity of the M$\ddot{o}$bius
function}, Ann. Inst. Fourier (Grenoble), {\bf 58} (2008), no. 6,
1863-1935.

\bibitem{GT} B. Green and T. Tao, {\it The quantitative behaviour of polynomial
orbits on nilmanifolds}, Annals of Math., to appear.

\bibitem{GT10} B. Green and T. Tao, \textit{Linear equations in primes},
Ann. of Math., {\bf 171} (2010), no. 3, 1753-1850.

\bibitem{HK05} B. Host and B. Kra, \textit{Nonconventional averages and
nilmanifolds}, Ann. of Math., {\bf 161} (2005), 398-488.

\bibitem{HK09} B. Host and B. Kra, \textit{Uniformity norms on $l^\infty$ and applications},
J. Anal. Math., {\bf 108} (2009), 219-276.

\bibitem{HK10} B. Host and B. Kra, \textit{Nil-Bohr sets of integers},
Ergod. Th. and Dynam. Sys., {\bf 31} (2011), 113-142.

\bibitem{HKM} B. Host, B. Kra and A. Maass, \textit{Nilsequences and a structure
theory for topological dynamical systems}, Adv. in Math., {\bf 224}
(2010), 103-129.

\bibitem{HM} B. Host and A. Maass, \textit{ Nilsyst\`emes d'ordre deux et
parall\'el\'epip\`edes}, Bull. Soc. Math. France, {\bf 135} (2007),
367-405.

\bibitem{HLY2}W. Huang, H. Li and X. Ye, {\it Localization and dynamical Ramsey property},
preprint.

\bibitem{HLY} W. Huang, P. Lu and X. Ye, {\it Measure-theoretical
sensitivity and equicontinuity}, Israel J. of Math., {\bf 183}
(2011), 233-284.

\bibitem{HSY} W. Huang, S. Shao and X. Ye, {\it Nil-Bohr$_0$ sets,
Poincar\'e recurrence and generalized polynomials}, preprint
arXiv:1109.3636.

\bibitem{Kaz} Y. Katznelson, {\it Chromatic numbers of Cayley graphs
on $\Z$ and recurrence}, Paul Erd\'{o}s and his mathematics
(Budapest, 1999), Combinatorics, {\bf 21}(2001), 211-219.

\bibitem{Ke} H. B. Keynes, {\it The structure of weakly mixing minimal transformation groups},
Illinois J. Math., {\bf 15} (1971), 475-489.

\bibitem{K} I. K$\check{\rm r}$\'i$\check{\rm z}$,
\textit{Large independent sets in shift-invariant graphs. Solution
of Bergelson's problem}, Graphs Combin., {\bf 3} (1987), 145-158.



\bibitem{SY} S. Shao and X.D. Ye, \textit{Regionally proximal relation of order $d$ is an
equivalence one for minimal systems and a combinatorial
consequence}, preprint, arXiv:1007.0189.

\bibitem{ShenYi} W.X. Shen and Y.F. Yi, \textit{ Almost automorphic and almost periodic
dynamics in skew-product semiflows}, Mem. Amer. Math. Soc., {\bf
136} (1998), no. 647, x+93 pp.

\bibitem{V65} W. A. Veech, \textit{ Almost automorphic functions on groups},
Amer. J. Math., {\bf 87}(1965), 719-751.

\bibitem{V68} W. A. Veech, \textit{ The equicontinuous structure relation for
minimal Abelian transformation groups}, Amer. J. Math., {\bf
90}(1968), 723-732.


\bibitem{V77} W. A. Veech, \textit{ Topological systems}, Bull. Amer. Math. Soc.,
{\bf 83}(1977), 775-830.

\bibitem{Vr} J. de Vries, \textit{Elements of Topological Dynamics},
Kluwer Academic Publishers (1993), Dordrecht.


\bibitem{Z} T. Ziegler, \textit{Universal characteristic factors and Furstenberg
averages}. J. Amer. Math. Soc., {\bf 20} (2007), no. 1, 53-97.

\end{thebibliography}
\end{document}